\documentclass[11pt,a4paper]{amsart}

\usepackage[centering]{geometry}
\usepackage{amsfonts}
\usepackage{amsmath}
\usepackage{amssymb}
\usepackage{amsthm}
\usepackage{cite}
\usepackage[T1]{fontenc}
\usepackage[utf8]{inputenc}
\usepackage{lmodern}
\usepackage{mathrsfs}
\usepackage{mathtools}
\usepackage{microtype}
\usepackage{enumitem} 

\usepackage{xcolor}
\usepackage[hidelinks,hypertexnames=false]{hyperref}
\usepackage[nameinlink]{cleveref}                      

\DeclareMathOperator{\Dom}{\mathcal{D}}
\DeclareMathOperator{\Lin}{Lin}
\DeclareMathOperator{\Rng}{Rng}
\DeclareMathOperator{\realpart}{Re}
\DeclareMathOperator{\sech}{sech}
\DeclareMathOperator{\signum}{sgn}
\DeclareMathOperator{\linspan}{span}
\DeclareMathOperator{\spectrum}{spec}

\newcommand{\nonnorm}{\nabla_{\hspace{-0.075cm} \bot}}
\newcommand{\nontan}{\nabla_{\hspace{-0.075cm} \top}}

\newcommand{\Id}{\mathrm{Id}}
\newcommand*{\placeholder}{{\mkern 2mu\cdot\mkern 2mu}}

\DeclarePairedDelimiter{\abs}{\lvert}{\rvert}
\DeclarePairedDelimiter{\jb}{\langle}{\rangle}
\DeclarePairedDelimiter{\norm}{\lVert}{\rVert}
\DeclarePairedDelimiter{\brac}{\lbrace}{\rbrace}
\DeclarePairedDelimiter{\brak}{\lbrack}{\rbrack}
\DeclarePairedDelimiter{\parn}{\lparen}{\rparen}    

\newcommand{\C}{\mathbb{C}}
\newcommand{\R}{\mathbb{R}}
\newcommand{\N}{\mathbb{N}}

\newcommand{\Xspace}{\mathbb{X}}

\newcommand{\Wspace}{\mathbb{W}}
\newcommand{\Vspace}{\mathbb{V}}
\newcommand{\cinterval}{\mathcal{I}}

\newcommand{\augV}{\mathcal{V}_c^{\mathrm{aug}}}
\newcommand{\eng}{E}
\newcommand{\Hc}{H_c}
\newcommand{\augHam}{E_c}
\newcommand{\mom}{P}
\newcommand{\potE}{V}
\newcommand{\kinE}{K}
\newcommand{\Pmanifold}{\mathcal{M}}
\newcommand{\DN}{G}

\newcommand{\dUdc}{\frac{dU_c}{dc}}
\newcommand{\nbhdO}{\mathcal{O}}
\newcommand{\tube}{\mathcal{U}}

\newcommand{\fund}{\Gamma}

\newcommand{\regindex}{k}
\newcommand{\regW}{5/2}
\newcommand{\regV}{1}

\newcommand{\sTheta}{\Theta|_S}

\numberwithin{equation}{section}

\theoremstyle{plain} 
\newtheorem{theorem}{Theorem}[section] 
\newtheorem{corollary}[theorem]{Corollary}
 
\newtheorem{lemma}[theorem]{Lemma}
\theoremstyle{remark}
\newtheorem{remark}[theorem]{Remark}
\newtheorem{assumption}{Assumption}

\title[Point vortex stability]{On the stability of solitary water waves with a point vortex}

\author[K. Varholm]{Kristoffer Varholm}
\address{Department of Mathematical Sciences, Norwegian University of Science and Technology, 7491 Trondheim, Norway}
\email{kristoffer.varholm@ntnu.no}

\author[E. Wahl\'en]{Erik Wahl\'en}
\address{Centre for Mathematical Sciences, Lund University,  PO Box 118, 22100 Lund, Sweden}
\email{erik.wahlen@math.lu.se}

\author[S. Walsh]{Samuel Walsh}
\address{Department of Mathematics, University of Missouri, Columbia, MO 65211, USA} 
\email{walshsa@missouri.edu} 

\date{\today}


\begin{document}

\begin{abstract}
    This paper investigates the stability of traveling wave solutions to the free boundary Euler equations with a submerged point vortex. We prove that sufficiently small-amplitude waves with small enough vortex strength are conditionally orbitally stable.  In the process of obtaining this result, we develop a quite general stability/instability theory for bound state solutions of a large class of infinite-dimensional Hamiltonian systems in the presence of symmetry.  This is in the spirit of the seminal work of Grillakis, Shatah, and Strauss (GSS) \cite{grillakis1987stability1}, but with hypotheses that are relaxed in a number of ways necessary for the point vortex system, and for other hydrodynamical applications more broadly.  In particular, we are able to allow the Poisson map to have merely dense range, as opposed to being surjective, and to be state-dependent.
    
    As a second application of the general theory, we consider a family of nonlinear dispersive PDEs that includes the generalized Korteweg--de Vries (KdV) and Benjamin--Ono equations.  The stability or instability of solitary waves for these systems has been studied extensively, notably by Bona, Souganidis, and Strauss \cite{bona1987stability}, who used a modification of the GSS method.  We provide a new, more direct proof of these results, as a straightforward consequence of our abstract theory.  At the same time, we allow fractional dispersion, and obtain a new instability result for fractional KdV.
\end{abstract}

\maketitle
\setcounter{tocdepth}{1}
\tableofcontents

\section{Introduction} \label{introduction section}
    The persistence of localized regions of vorticity is a remarkable feature of two-dimensional incompressible inviscid fluid motion.  For instance, high Reynolds number flow over an immersed body may produce a wake of shed vortices outside of which the velocity field is largely irrotational.  While the small-scale structure of these regions can be quite intricate, their large-scale movement is well predicted by the so-called Helmholtz--Kirchhoff point vortex model, so long as they remain sufficiently isolated.  The stability of various configurations of point vortices in a fixed domain has therefore been the subject of extensive study since the early work of Poincar\'e \cite{poincare1893theorie}. In this paper, we are interested in point vortices carried by water waves.   Unlike the fixed domain case, this will involve understanding the subtle dynamical implications of wave--vortex interactions.  Our main results concern the orbital stability of small-amplitude solitary waves with a single point vortex.
    
    To state things more precisely, by ``water'' we mean an incompressible, homogeneous, and inviscid fluid occupying a time-dependent domain $\Omega_t \subset \R^2$.  For simplicity, assume that at time $t \geq 0$,  $\Omega_t$ consists of the (unbounded) region lying below the graph of a function $\eta = \eta(t,x_1)$, and above $\Omega_t$ is vacuum.  This is a free boundary problem, in the sense that $\eta$ is not prescribed, but evolves dynamically.  
    
    Let ${v} = v(t,\placeholder) \colon \Omega_t \to \R^2$ denote the fluid velocity at time $t \geq 0$.  The \emph{vorticity} is defined to be the quantity
    \begin{equation}
        \label{definition of vorticity}
        \omega \coloneqq \nabla^\perp \cdot v, \qquad \nabla^\perp \coloneqq (-\partial_{x_2},\partial_{x_1}),
    \end{equation}
    measuring the circulation density of the fluid. Mathematically, a \emph{point vortex} describes the situation where $\omega = \epsilon \delta_{\bar{x}(t)}$, a weighted Dirac measure supported at $\bar x = \bar x(t) \in \Omega_t$.  We call $\epsilon$ the \emph{vortex strength} and $\bar x$  the \emph{vortex center}.   It is fairly easy to see that this is not a valid measure-valued solution of the vorticity equation, as the advection term ${v} \cdot \nabla \omega$ has no distributional meaning.   Instead,  we ask only that the velocity field be a weak solution to the incompressible irrotational Euler equations away from the vortex center. That is,
    \begin{subequations} \label{intro point vortex problem}
    \begin{equation}
        \label{euler equations}
        \left\{
        \begin{aligned}
            \partial_t {v}  + \nabla \cdot ( {v} \otimes {v} ) & = -\nabla p - g{e}_2 & & \qquad \text{in } \Omega_t \setminus \brac{\bar{x}(t)}, \\
            \omega & = \epsilon \delta_{\bar{x}(t)} & & \qquad \text{in } \Omega_t, \\
            \nabla \cdot {v} & = 0 & & \qquad \text{in } \Omega_t,
        \end{aligned}
        \right. 
    \end{equation}
    with each of these holding in the sense of distributions.  Here $p = p(t,\placeholder) \colon \Omega_t \to \R$ is the pressure and $g > 0$ is the acceleration due to gravity. We consider the finite \emph{excess} energy case where $v(t) \in L_{\text{loc}}^1(\Omega_t) \cap L^2(\Omega_t\setminus U_t)$ for every neighborhood $U_t \ni \bar{x}(t)$. The motion of the point vortex is taken to be governed by the Helmholtz--Kirchhoff model 
    \begin{equation}
        \label{euler point vortex motion}
        \partial_t \bar{x} = \parn*{v-\frac{\epsilon}{2\pi} \nabla^\perp  \log{\abs{\placeholder -\bar{x}}}}\bigg|_{\bar{x}},
    \end{equation}
    where the subtracted term is the velocity field generated by the point vortex.  Thus \eqref{euler point vortex motion} states that the vortex center does not self-advect, but rather is transported only by the irrotational part of the fluid velocity field.
    
    Finally, the evolution of the free boundary is coupled to that of the fluid by the requirements that
    \begin{equation}
        \label{euler boundary condition}
        \begin{gathered}
            \partial_t \eta = (-\partial_{x_1}\eta,1) \cdot v,\\
            p = -b \partial_{x_1}\parn*{\frac{\partial_{x_1}\eta}{\jb{\partial_{x_1}\eta}}},
        \end{gathered}
    \end{equation}
    \end{subequations}
    on the interface $S_t \coloneqq \partial \Omega_t$, and where $b > 0$ is the coefficient of surface tension, and $\jb{\placeholder} \coloneqq \sqrt{1+\abs{\placeholder}^2}$ is the useful Japanese bracket.  The first of the requirement in \eqref{euler boundary condition} is the kinematic condition, linking the surface to the velocity field.  The second is the dynamic condition, which states that the pressure deviates from atmospheric pressure (normalized here to $0$) in proportion to the signed curvature of the surface.  
    
    Point vortices have been studied in fluid mechanics for centuries.  The specific model \eqref{euler equations}--\eqref{euler point vortex motion} was first proposed by Helmholtz \cite{helmholtz1858integrale} and Kirchhoff \cite{kirchhoff1876vorlesungen} for incompressible fluids in a fixed domain.  Later, Marchioro and Pulvirenti (see \cite{marchioro1993vortices} and \cite[Chapter 4]{marchioro1994book}) offered a rigorous justification by proving that \eqref{euler equations}--\eqref{euler point vortex motion} is the limiting equation governing the motion of vortex patch solutions of the Euler equations as the diameter of the patch approaches $0$.  Another derivation was given by Gallay \cite{gallay2011interaction}, who showed that the system can be obtained as the vanishing viscosity limit for smooth solutions of the Navier--Stokes equation with increasingly concentrated vorticity.   The recent work of Glass, Munnier, and Sueur \cite{glass2018point} provides a second physical interpretation:  they prove that the Helmholtz--Kirchhoff system governs irrotational incompressible inviscid flow around an immersed rigid body, with a fixed circulation around the body, in the limit where the body shrinks to a point in a specific way.
    
    The primary objective in this paper is to study the stability of steady solutions of the water wave with a point vortex problem \eqref{intro point vortex problem}.  An existence theory for waves of this type was given by Shatah, Walsh, and Zeng \cite{shatah2013travelling}.  The analogous problem for capillary-gravity waves in finite-depth water was recently considered by Varholm \cite{varholm2016solitary}, and for gravity waves by Ter-Krikorov \cite{terkrkorov1958vortex} and Filippov \cite{filippov1960vortex,filippov1961motion}. These are among the very few examples of exact steady water waves with localized vorticity currently available.  Numerical studies of water waves with a point vortex have been carried out in \cite{elcrat2006free,curtis2017vortex,doak2017solitary}, for example.
    
    Stated informally, our main result is as follows.  First, observe that in a neighborhood of $S_t$, the velocity field $v$ can be decomposed as 
    \[
        v = \nabla \Phi + \epsilon \nabla \Theta,
    \]
    where $\Phi(t,\placeholder)$ is harmonic in $\Omega_t$, and $\Theta$ is an explicit function depending on $\bar x$ that captures the contribution of the point vortex; see Section~\ref{nonlocal formulation section}.  The system \eqref{intro point vortex problem} can then be reformulated as an equation for $u = ( \eta, \varphi, \bar x)$, where 
    \[
        \varphi = \varphi(t,x_1) \coloneqq \Phi(t, x_1, \eta(t, x_1)).
    \]
    A \emph{solitary wave} in this setting corresponds to a solution of the form
    \[
        u(t, x_1) = (\eta^c(x_1-ct), \varphi^c(x_1-ct), \bar x^c + cte_1),
    \]
    for some spatially localized $(\eta^c, \varphi^c, \bar x^c)$ and wave speed $c \in \R$.  
    
    \begin{theorem}[Main result]
        \label{main theorem}
        Every symmetric solitary capillary-gravity water wave with a point vortex $(\eta^c, \varphi^c, \bar x^c)$ having $(\eta^c, \varphi^c)$, $c$, and  $\epsilon$ sufficiently small is \emph{conditionally orbitally stable} in the following sense.  For all $R > 0$ and $\rho > 0$, there exists $\rho_0> 0$ such that, if $(\eta, \varphi, \bar x)$ is any solution defined on a time interval $[0, t_0)$, obeying a bound
        \begin{equation}
            \label{main theorem W bound}
            \sup_{t \in [0, t_0)} \parn*{\norm{\eta(t)}_{H^{3+}} + \norm{\varphi(t)}_{\dot{H}^{\frac{5}{2}+} \cap \dot{H}^{\frac{1}{2}}} + \abs{\bar x_2(t)} } < R, 
        \end{equation}
        and having initial data satisfying
        \[
            \norm{ \eta(0) - \eta^c }_{H^{1}} + \norm{ \varphi(0)  - \varphi^c }_{\dot{H}^{\frac{1}{2}}}+ \abs{ \bar x(0) - \bar x^c } < \rho_0,
        \]
        then 
        \begin{equation}
            \label{main theorem X bound}
            \adjustlimits\sup_{t \in [0, t_0)}\inf_{s \in \R} \left( \| \eta(t, \placeholder -s) - \eta^c  \|_{H^{1}} +  \| \varphi(t, \placeholder - s)  - \varphi^c \|_{ \dot{H}^{\frac{1}{2}}} + | \bar x(t) + s e_1 - \bar x^c | \right) < \rho.
        \end{equation}
    \end{theorem}
    
    A more precise version is given in Theorem~\ref{main point vortex theorem}. Several remarks are in order. Orbital here refers to the fact that we are controlling the distance to the family of translates of the steady wave; this is natural given the translation-invariant nature of the problem. It is also important to note that $\rho_0$ above is independent of $t_0$, and hence the conclusion of Theorem~\ref{main theorem} is much stronger than just continuity of the solution map at $(\eta^c, \varphi^c, \bar x^c)$. Indeed, for a global-in-time solution, this gives orbital stability in the classical sense.  The norm occurring in \eqref{main theorem W bound} represents the lowest regularity in which a local well-posedness theory has been established for irrotational capillary-gravity waves \cite{alazard2011surface}.  On the other hand, the norm in \eqref{main theorem X bound} is associated to the physical energy for the system, which we will discuss shortly.
    
    Our approach is to rewrite \eqref{intro point vortex problem} as an infinite-dimensional Hamiltonian system of the general form
    \[
        \frac{du}{dt} = J(u) DE(u),
    \]
    with $u$ appropriate Banach space. Here, $E$ is a functional (the energy), and $J$ is a state-dependent skew-adjoint operator (the Poisson map). A similar system was established formally by Rouhi and Wright \cite{rouhi1993hamiltonian}; we use a slightly different version, and give a rigorous proof in Section~\ref{point vortex hamiltonian section}.  
    
    As the entire problem is invariant under translation, there is a conserved momentum functional $P = P(u)$.  A natural strategy for analyzing the (orbital) stability of bound states in abstract Hamiltonian systems with symmetries is to use the \emph{energy-momentum method} first introduced by Benjamin \cite{benjamin1972stability}.  In brief, this method involves constructing a Lyapunov functional using a carefully chosen combination of $E$ and $P$.  Actually carrying out this argument, however, can be quite challenging.  Over three decades ago, Grillakis, Shatah, and Strauss \cite{grillakis1987stability1} introduced a powerful machinery --- now commonly referred to as the GSS method --- which essentially reduced these many difficulties down to discerning the convexity or concavity of a single scalar-valued quantity called the moment of instability.  
    
    Not surprisingly, this paper had an enormous impact on the field and generated a great deal of research activity.    However, the hypotheses of GSS limit somewhat its applicability to infinite-dimensional Hamiltonians with more complicated structure.  For instance, they require that the operator $J$ is surjective, and independent of the state $u$.  But, recall that the Poisson map for the Korteweg--de Vries (KdV) equation is $\partial_x$, which is not surjective in the natural class of spaces. In fact, for water waves with a point vortex \eqref{intro point vortex problem}, we will see that $J$ is neither independent of state, nor surjective.
    
    There is also a somewhat practical issue with the functional analytic setting.  Consider for a moment the irrotational case.  GSS supposes that the Cauchy problem is globally well-posed in the energy space.  But, as remarked above, the local well-posedness of the gravity water wave problem  with surface tension proved by Burq, Alazard, and Zuily in \cite{alazard2011surface} takes $\eta(t) \in H^{3+}$ and $\varphi(t) \in  \dot H^{5/2+} \cap \dot H^{1/2}$.  On the other hand, the kinetic energy is given by the much rougher $\| v \|_{L^2}^2$, and the potential energy is equivalent to $\| \eta \|_{H^1}^2$.  Moreover, writing the kinetic energy in terms of $(\eta, \varphi)$ yields 
    \[
        \norm{ v }_{L^2(\Omega_t)}^2 = \frac{1}{2}\int_{\R} \varphi \DN(\eta) \varphi \, dx_1,
    \]
    where $\DN(\eta)$ is the Dirichlet--Neumann operator; see the discussion in Section~\ref{nonlocal formulation section}.  For this energy to be smooth as a functional of $(\eta, \varphi)$ in the Sobolev setting, one must have that $\eta \in H^{3/2+} \hookrightarrow W^{1,\infty}$.   In effect, then, there are three levels of regularity:  a rough space in which the physical energy is defined, an intermediate space where the energy functional is smooth, and a yet higher regularity space where we can hope to have well-posedness.  This situation is exceedingly common in the analysis of quasilinear equations.  Indeed, it is the natural by-product of so-called higher-order energy estimates, which are among the most basic and widespread tools in nonlinear PDE theory.
    
    With that in mind, as one of the primary contributions of this paper, we introduce a new abstract stability/instability-result in the spirit of GSS, but with relaxed assumptions; making it directly applicable to problems such as \eqref{intro point vortex problem}.  Specifically, we allow for a large class of state-dependent Poisson maps $J = J(u)$, and essentially only require that $J$ is injective with dense range.  Moreover, the entire theory is formulated in a scale of Banach spaces, offering a simple way to accommodate gaps between the necessary regularity levels for the energy. Finally, in view of the point vortex problem, we allow the symmetry group to be merely affine.
    
    There are a number of new assumptions and technical conditions, but the main conclusion is the same as that of GSS: stability or instability of the bound state hinges on the sign of a scalar quantity. Because of the mismatch in spaces, our results are conditional in the sense that they only hold on a time interval in which the solutions of the problem exist and their growth is controlled.  Using this general theory, we are then able to address the question of stability of traveling water waves with a point vortex and prove Theorem~\ref{main theorem}. Finally, we also consider a  further application of this same framework to KdV, and related dispersive model equations. 
    
    One of the main inspirations for this paper is Mielke's  work on conditional energetic stability of irrotational solitary waves on water of finite depth with strong surface tension  \cite{mielke2002energetic}, in which he also had to modify the GSS method to deal with the mismatch between well-posedness and energy spaces. While our basic strategy is the same, we make the additional effort of formulating a general theory which also deals with instability. On a technical level, the presence of the point vortex requires a number of non-trivial modifications. Mielke's work was followed by a series of papers proving the existence and conditional stability of different families of solitary water waves by a variational approach in which the waves are constructed using the direct method of the calculus of variations as minimizers of the energy subject to the constraint of fixed momentum. The stability of the set of minimizers then follows directly from classical arguments by Cazenave and Lions. In particular, Buffoni \cite{buffoni2004stability} considered solitary waves on finite depth with strong surface tension. He also obtained partial results in the case of finite depth and weak surface tension, as well as in the case of infinite depth \cite{buffoni2005conditional, Buffoni09}, which were later completed by Groves and  Wahl\'en \cite{groves2010existence, GrovesWahlen11}. More recently, these authors also extended the method to solitary water waves with constant vorticity \cite{groves2015existence}. Similar to the present study, the Hamiltonian formulation is non-canonical in that case. It is likely that direct variational methods could be used also in the presence of point vortices.
    
    \subsection*{Plan of the article}
        In Section~\ref{abstract formulation section}, we give a detailed description of our results regarding conditional orbital stability and instability of bound states in abstract Hamiltonian systems with symmetry. Our main result on orbital stability is Theorem~\ref{abstract stability theorem}, which is proved in Section~\ref{abstract stability section}.  The unstable case is addressed in Theorem~\ref{abstract instability theorem}, whose proof is carried out in Section~\ref{abstract instability}.  
        
        We return to the water wave with a point vortex problem in Section~\ref{point vortex formulation section}, where it is shown that \eqref{intro point vortex problem} can be reformulated as an infinite-dimensional Hamiltonian system of the type covered by the general theory.  In Sections~\ref{point vortex spectrum section}, we characterize the spectrum of the so-called linearized augmented Hamiltonian at a solitary wave, which is used to prove our main result: small-amplitude and small vorticity symmetric solitary capillary-gravity water waves with a point vortex are conditionally orbitally stable; see Theorem~\ref{main point vortex theorem}.
        
        To demonstrate the broader implications of the general theory, we consider a large family of nonlinear dispersive PDEs in Section~\ref{dispersive section}. These serve as approximate models for water waves, and include both the KdV and Benjamin--Ono (BO) equations.  Because the corresponding $J$ is not surjective between the relevant spaces, these equations lie outside the GSS framework. In \cite{bona1987stability}, Bona, Souganidis, and Strauss overcame this difficulty by supplementing the basic approach of GSS with a consideration of the mass.  On the other hand, the general theory we develop in the present paper can be directly applied to this family of equations, meaning we are able to give a new proof of the Bona--Souganidis--Strauss theorem as a straightforward application.  In fact, this also furnishes new instability results for fractional KdV; see Theorem~\ref{bss stability theorem}

\section{General setting and main results}
    \label{abstract formulation section}
    
    \subsection{Formulation and hypotheses}
        \label{abstract formulation assumptions section}
        
        We will work with a scale of spaces
        \[
            \Wspace \hookrightarrow \Vspace \hookrightarrow \Xspace, 
        \]
        where $\Xspace$ is a real Hilbert space, while $\Vspace$ and $\Wspace$ are reflexive Banach spaces. The inner product on $\Xspace$ will be denoted by $(\placeholder, \placeholder)_{\Xspace}$, and the corresponding norm by $\norm{\placeholder}_{\Xspace}$.  Likewise, let $\norm{\placeholder}_{\Vspace}$ and $\norm{\placeholder}_{\Wspace}$ be the norms for $\Vspace$ and $\Wspace$, respectively.  We write $\Xspace^*$ for the (continuous) dual of $\Xspace$, which is naturally isomorphic to $\Xspace$ via the mapping $I \colon \Xspace \to \Xspace^*$ taking $u \in \Xspace$ to $(u, \placeholder)_{\Xspace} \in \Xspace^*$. We will not make this identification here, but rather use $I$ explicitly.  On the other hand, we will simply identify $\Xspace^{**}$ with $\Xspace$, and likewise for $\Vspace$ and $\Wspace$.  The pairing of $\Xspace$ and $\Xspace^*$ we denote by $\jb{\placeholder, \placeholder}_{\Xspace^* \times \Xspace}$, while $\jb{ \placeholder, \placeholder }_{\Wspace^* \times \Wspace}$ is the pairing between $\Wspace^*$ and $\Wspace$; when there is no risk of confusion, we will omit the subscript.   
        
        Intuitively, $\Xspace$ is the energy space for the system under consideration.  This is where the Hamiltonian structure will be formulated, and is the natural setting for analyzing the spectrum. On the other hand, $\Vspace$ is a space where the conserved quantities are smooth. Finally, we think of $\Wspace$ as a ``well-posedness space'', with the norm coming from higher-order energy estimates used to prove that the Cauchy problem is at least locally well-posed in time. The norm on $\Wspace$ also plays the secondary role of allowing us to get control over $\Vspace$ via interpolation.  More precisely, we require the following:
        
        \begin{assumption}[Spaces]
            \label{abstract interpolation assumption}
            Let $\Xspace$, $\Vspace$, and $\Wspace$ be given as above.  Assume that there exist constants $\theta \in (0,1]$ and $C> 0$ such that
            \begin{equation}
                \label{abstract interpolation inequality}
                \norm{u}_{\Vspace}^3 \leq C \norm{u}_{\Xspace}^{2+\theta} \norm{u}_{\Wspace}^{1-\theta}
            \end{equation}
            for all $u \in \Wspace$.
        \end{assumption}
        \begin{remark}
            \label{continuity remark}
            A useful consequence of \eqref{abstract interpolation inequality} is that, if $F \in C^3(\Vspace; \R)$, and $B \subset \Wspace$ is a bounded set, then
            \[
               F(x + h) -F(x)  =  \jb{DF(x),h} + \frac{1}{2}\jb{D^2 F(x)h,h} + O(\norm{h}_{\Xspace}^{2+\theta})
            \]
            for $x \in \Vspace$ and $h \in B$.
        \end{remark}
        
        It is often necessary to restrict attention to some smaller subset of these spaces in order to ensure that the problem is well-defined.  For example, in the case of the traveling waves with a point vortex, there must be a positive separation between the vortex center and the air--sea interface.  Abstractly, we will handle these types of situations by introducing an open set $\nbhdO \subset \Xspace$, where solutions must live.
        
        Suppose that $\hat{J} \colon D(J) \subset \Xspace^* \to \Xspace$ is a closed linear operator, and that we for each $u \in \nbhdO \cap \Vspace$ have a bounded linear operator $B(u) \in \Lin(\Xspace)$. We endow $\Xspace$ with symplectic structure in the form of the state-dependent Poisson map
        \begin{equation}
            \label{abstract state dependent poisson map}
            J(u) \coloneqq B(u)\hat{J},
        \end{equation}
        which is required to satisfy a number of hypotheses.
        
        \begin{assumption}[Poisson map]
            \label{abstract symplectic assumption}
            \leavevmode
            \begin{enumerate}[label=\rm(\roman*)]
                \item The domain $\Dom(\hat{J})$ is dense in $\Xspace^*$.     \label{J densely defined assumption}
                \item \label{J injectivity assumption} $\hat{J}$ is injective.
                \item \label{B bijective assumption} For each $u \in \nbhdO \cap \Vspace$, the operator $B(u)$ is bijective.
                \item \label{regularity of J} The map $u \mapsto B(u)$ is of class $C^1(\nbhdO \cap \Vspace;\Lin(\Xspace)) \cap C^1(\nbhdO \cap \Wspace; \Lin(\Wspace))$.
                \item For each $u \in \nbhdO \cap \Vspace$, $J(u)$ is skew-adjoint in the sense that
                \[
                \jb{J(u)v,w} = -\jb{v,J(u)w}
                \]
                for all $v,w \in \Dom(\hat{J})$.
            \end{enumerate}
        \end{assumption}
        \begin{remark}
            Note that this does not assume that $J(u)$ is surjective, which is a significant departure from GSS.  Below, we will require something slightly stronger than that the range of $J(u)$ is dense in $\Xspace$.  
        \end{remark}
        
        The main object of interest for this work is the abstract Hamiltonian system 
        \begin{equation}
            \label{abstract Hamiltonian system}
            \frac{du}{dt} = J(u) D \eng(u), \qquad u|_{t=0} = u_0,
        \end{equation}
        where $\eng \in C^3(\nbhdO \cap \Vspace; \R)$ is the \emph{energy functional}.   In addition to the energy, we suppose that there is a second conserved quantity $\mom \in C^3(\nbhdO \cap \Vspace; \R)$, which we call the \emph{momentum}. In order to state what it means to be a solution of \eqref{abstract Hamiltonian system}, and to work with it in a meaningful way, we need to be able to view $D\eng(u)$ and $D\mom(u)$ as elements of $\Xspace^*$.
        
        \begin{assumption}[Derivative extension] \label{extend DP and DE assumption} There exist mappings $\nabla \eng, \nabla \mom \in C^0(\nbhdO \cap \Vspace;\Xspace^*)$ such that $\nabla \eng (u)$ and $\nabla \mom (u)$ are extensions of $D\eng(u)$ and $D\mom(u)$, respectively, for every $u \in \nbhdO \cap \Vspace$.
        \end{assumption}
        
        We say that $u  \in C^0([0,t_0);  \nbhdO \cap \Wspace)$ is a solution of \eqref{abstract Hamiltonian system} on the interval $[0, t_0)$ if
        \begin{equation}
            \label{weak abstract Hamiltonian system}
            \frac{d}{dt}\left\langle u(t), \, w \right\rangle = -\left\langle \nabla\eng(u(t)), \,  J(u(t)) w \right\rangle \qquad \text{ for all $w \in \Dom(\hat{J})$,}
        \end{equation}
        is satisfied in the distributional sense on $(0,t_0)$, the initial condition $u(0) = u_0$ is satisfied, and both $\eng$ and $\mom$ are conserved.
        
        Of particular importance is the situation where the system \eqref{abstract Hamiltonian system} is invariant with respect to a symmetry group. Specifically, we assume that there exists a one-parameter family of affine maps $T(s) \colon \Xspace \to \Xspace$, with linear part $dT(s)u \coloneqq T(s)u - T(s)0$, having the properties described below. We refer to \cite{goldstein1988semigroups} for a background on affine groups on Banach spaces.
        
        \begin{assumption}[Symmetry group]
            \label{abstract symmetry assumption}
            The symmetry group $T(\placeholder)$ satisfies the following.  
            \begin{enumerate}[label=\rm(\roman*)]
                \item \label{invariances} \emph{(Invariance)}
                    The neighborhood $\nbhdO$, and the subspaces $\Vspace$ and $\Wspace$, are all invariant under the symmetry group. Moreover, $I^{-1}\Dom(\hat{J})$ is invariant under the \emph{linear} symmetry group.
                \item \label{group flow property} \emph{(Flow property)}
                    We have $T(0) = dT(0) = \mathrm{Id}_\Xspace$, and for all $s, r \in \R$,
                    \[
                        T(s+r) = T(s)T(r), \qquad \text{and hence} \qquad dT(s+r) = dT(s)dT(r).
                    \] 
                \item \label{unitary assumption} \emph{(Unitarity)}
                    The linear part $dT(s)$ is a unitary operator on $\Xspace$, and an isometry on $\Vspace$ and $\Wspace$, for each $s \in \R$.
                \item \label{strong continuity} \emph{(Strong continuity)}
                    The symmetry group is strongly continuous on $\Xspace$, $\Vspace$, and $\Wspace$.
                \item \label{affine bound assumption} \emph{(Affine part)} 
                    The function $T(\placeholder)0$ belongs to $C^3(\R; \Wspace)$, and there exists an increasing function $\omega \colon [0,\infty) \to [0,\infty)$ such that
                    \[
                    \norm{T(s)0}_{\Wspace} \leq \omega(\norm{T(s)0}_{\Xspace}), \quad \text{for all } s \in \R.
                    \]
                \item \label{commutativity assumption} \emph{(Commutativity with $J$)}
                    For all $s \in \R$,
                    \begin{equation}
                        \label{abstract commutation identity}
                        \begin{aligned}
                            \hat{J}I dT(s) &= dT(s)\hat{J}I,\\
                            dT(s)B(u) &= B(T(s)u)dT(s), \quad \text{for all } u \in \nbhdO\cap \Vspace.
                        \end{aligned}
                    \end{equation}
                \item \label{T'(0) assumption} \emph{(Infinitesimal generator)}
                    The infinitesimal generator of $T$ is the affine mapping 
                    \[
                        T'(0)u=\lim_{s \to 0} \parn*{s^{-1}(T(s)u - u)} = dT'(0)u + T'(0)0,
                    \]
                    with dense domain $\Dom(T^\prime(0)) \subset \Xspace$ consisting of all $u\in \Xspace$ such that the limit exists in $\Xspace$ (note that $\Dom(T^\prime(0)) =\Dom(dT^\prime(0))$ by the first part of assumption \ref{affine bound assumption}). Similarly, we may speak of the dense subspaces $\Dom(T'(0)|_\Vspace) \subset \Vspace$ and $\Dom(T'(0)|_\Wspace) \subset \Wspace$ on which the limit exists in $\Vspace$ and $\Wspace$, respectively.
                    
                    We assume that $\nabla P (u) \in \Dom(\hat{J})$ for every $u \in \Dom(T'(0)|_\Vspace) \cap \nbhdO$, and that
                    \begin{equation}
                        \label{abstract T'(0) and P' identity}
                        T'(0)u = J(u)\nabla P (u)
                    \end{equation}
                    for all such $u$. Moreover, we assume that
                    \begin{equation}
                        \label{abstract derivative commutation identity}
                        \hat{J}IdT'(0) = dT'(0)\hat{J}I.
                    \end{equation}
                \item \label{range density} \emph{(Density)}
                    The subspace
                    \[
                        \Dom(T'(0)|_\Wspace) \cap \Rng{\hat{J}}
                    \]
                    is dense in $\Xspace$.
                \item \label{T conserves energy} \emph{(Conservation)}
                    For all $u \in \nbhdO \cap \Vspace$, the energy is conserved by flow of the symmetry group:
                    \begin{equation}
                        \label{abstract energy invariance under T}
                        \eng(u) = \eng(T(s)u), \qquad \text{for all } s \in \R.
                    \end{equation}
            \end{enumerate}
        \end{assumption}
        \begin{remark}
            \label{symmetry group remark}  
            There are some immediate consequences of the above assumptions. We can combine parts \ref{group flow property} and \ref{commutativity assumption}  to deduce that
            \[
                dT(s) J(u) I = J(T(s)u)IdT(s), \qquad \text{for all } s \in \R, \, u \in \nbhdO\cap \Vspace,
            \]
            and as a consequence of the unitarity of $dT(s)$, the operator $dT'(0)$ is skew-adjoint on $\Xspace$. Moreover, if $u \in \Dom(T'(0)|_\Vspace) \cap \nbhdO$, then $s \mapsto P(T(s)u)$ has derivative
            \[
                \jb{\nabla \mom(T(s)u),T'(0)T(s)u} = \jb{\nabla \mom(T(s)u),J(T(s)u)\nabla \mom (T(s)u)} = 0
            \]
            by \eqref{abstract T'(0) and P' identity} and the skew-adjointness of $J(T(s)u)$. Thus, by density of $\Dom(T'(0)|_\Vspace)$ in $\Vspace$, the flow of the symmetry group also conserves the momentum for all $u \in \nbhdO \cap \Vspace$:
            \begin{equation}
                \label{abstract momentum invariance under T}
                \mom(u) = \mom(T(s)u), \qquad \text{for all } s \in \R.
            \end{equation}
            
        \end{remark}
        
        We say that $u \in C^1(\R; \nbhdO \cap \Wspace)$ is a \emph{bound state} of the Hamiltonian system \eqref{abstract Hamiltonian system} provided that it is a solution of the form
        \[
            u(t) = T(c t) U_c,
        \]
        for some $c \in \R$ and $U_c \in \nbhdO \cap \Wspace$. We will also refer to $U_c$ itself as a bound state. If $T$ represents translation, then bound states correspond to the familiar notion of traveling waves, such as the ones we will study later. For the general setting, we take it as given that an analogous family is available:
        
        \begin{assumption}[Bound states]
            \label{bound state assumption}
            There exists a one-parameter family of bound state solutions $\{ U_c :  c \in \cinterval \}$, where $\cinterval \subset \R$ is a non-empty open interval, to the Hamiltonian system \eqref{abstract Hamiltonian system}. The family enjoys the following properties.
            \begin{enumerate}[label=\rm(\roman*)]
                \item \label{bound states improved regularity}
                    The mapping $c \in \cinterval    \mapsto U_c \in \nbhdO \cap \Wspace$ is $C^1$.
                \item \label{bound state domain assumption}
                    For all $c \in \cinterval$, 
                    \begin{equation}
                        \label{technical bound state assumption}
                        U_c \in \Dom(T'''(0)) \cap \Dom(\hat{J}I T'(0)),
                    \end{equation}
                    and
                    \begin{equation}
                        \label{second technical bound state assumption}
                        U_c, \hat{J}IT'(0)U_c \in \Dom(T'(0)|_\Wspace).
                    \end{equation}
                \item \label{bound state non-degeneracy}
                    The non-degeneracy condition $T'(0) U_c \neq 0$ holds for every $c \in \cinterval$. Equivalently, due to \eqref{abstract T'(0) and P' identity}, $U_c$ is never a critical point of the momentum.
                \item \label{non-periodic bound state assumption}
                    Either $s\mapsto T(s) U_c$ is periodic, or $\liminf_{|s| \to \infty} \norm{ T(s) U_c - U_c }_{\Xspace} > 0$.
            \end{enumerate}
        \end{assumption}
        
        Observe that, due to \eqref{abstract energy invariance under T} and \eqref{abstract momentum invariance under T}, the energy and momentum of $T(s) U_c$ are independent of $s$.  For a fixed parameter $c$, the corresponding \emph{augmented Hamiltonian} is the functional $\augHam \in C^3(\Vspace \cap \nbhdO; \R)$ defined by 
        \[ \augHam(u) \coloneqq \eng(u) - c \mom(u).\]
        Assumption~\ref{bound state assumption} ensures that $U_c \in \Dom(T^\prime(0))$, and so it follows from \eqref{weak abstract Hamiltonian system}, \eqref{abstract T'(0) and P' identity}, and Assumption~\ref{abstract symplectic assumption} that
        \begin{equation}
            \label{abstract stationary equation}
            D\augHam(U_c) = D\eng(U_c) - c D \mom(U_c) = 0,
        \end{equation}
        meaning $U_c$ is a critical point of $\augHam$. Due to this observation, we can think of each of the bound state $U_c$ as being a critical point of the energy with the constraint of a fixed momentum, with the wave speed $c$ arising naturally as a Lagrange multiplier.  Also, differentiating \eqref{abstract stationary equation} with respect to $c$ reveals that 
        \begin{equation}
            \label{D^2 E = DP relation}
            \jb*{ D^2 \augHam(U_c) \frac{dU_c}{dc}, \placeholder } = \jb*{D\mom(U_c), \placeholder }.
        \end{equation}
        
        Commonly in applications, the bound states sit at a saddle point of the energy.  That is, the second derivative of the augmented Hamiltonian at $U_c$ has a single simple negative (real) eigenvalue, a $0$ eigenvalue generated by the symmetry group, and the rest of the spectrum lies along the positive real axis; bounded uniformly away from the origin.  This is the basic setting of the problem considered in Grillakis, Shatah, and Strauss \cite{grillakis1987stability1}, and it is precisely what we will encounter in our study of water waves later.  We therefore make the following hypotheses about the configuration of the spectrum for the general theory.
        
        \begin{assumption}[Spectrum]
            \label{spectral assumptions}
            The operator $D^2 \augHam(U_c) \in \Lin(\Vspace,\Vspace^*) $ extends uniquely to a bounded linear operator $\Hc \colon \Xspace \to \Xspace^*$ such that:
            \begin{enumerate}[label=\rm(\roman*)]
                \item \label{extensibility assumption}
                    $I^{-1} \Hc$ is self-adjoint on $\Xspace$.
                \item \label{spectrum config assumption}
                    The spectrum of $I^{-1} \Hc$ satisfies 
                    \begin{equation}
                        \label{spectrum of Hc}
                        \spectrum{(I^{-1} \Hc)} = \brac{-\mu_c^2 } \cup \brac{ 0 } \cup \Sigma_c,
                    \end{equation}
                    where $-\mu_c^2 < 0$ is a simple eigenvalue corresponding to a unit eigenvector $\chi_c$, $0$ is a simple eigenvalue generated by $T$, and $\Sigma_c \subset (0,\infty)$ is bounded away from $0$.  
            \end{enumerate}
        \end{assumption}
    
    \subsection{Main results on stability and instability}
        The central question we wish to address is whether the bound states of \Cref{spectral assumptions} are stable or unstable.  As there is an underlying invariance with respect to the group $T$, it is most natural to understand stability and instability in the orbital sense.   For any $U \in \Xspace$, we call the set $\brac{ T(s) U : s \in \R}$ the \emph{$U$-orbit generated by $T$}.  Formally speaking, $U_c$ is \emph{orbitally stable} provided that any solution to the Cauchy problem that is initially close enough to the $U_c$-orbit generated by $T$ (in the $\Xspace$ norm) remains near the orbit for all time.  Conversely, \emph{orbital instability} describes the situation where there exists initial data arbitrarily close to the $U_c$-orbit that nevertheless leaves some neighborhood of the orbit in finite time.   
        
        Making these concepts rigorous for the problem at hand is complicated both by the lack of a global well-posedness theory for the Cauchy problem \eqref{abstract Hamiltonian system}, and especially the mismatch of the energy and well-posedness spaces.  For that reason, all of our results will necessarily be \emph{conditional} in that they will hold only so long as we know the solution exists, and that its growth in $\Wspace$ is controllable.  
        
        The \emph{moment of instability}, which we call $d$, is the scalar-valued function that results from evaluating the augmented Hamiltonian along the family of bound states:
        \begin{equation}
            \label{abstract d definition}
            d(c) \coloneqq \augHam(U_c) = \eng(U_c) - c \mom(U_c).
        \end{equation} 
        Note that because each bound state $U_c$ is a critical point of the augmented Hamiltonian, differentiating $d$ gives the identity 
        \begin{equation}
            \label{abstract d' identity}
            d'(c) = \jb*{ D \augHam(U_c), \frac{dU_c}{dc}} - \mom(U_c) = -\mom(U_c),
        \end{equation}
        and differentiating once more yields
        \begin{equation}
            \label{abstract d'' identity}
            d''(c) = -\jb*{D\mom(U_c), \frac{dU_c}{dc} } = -\jb*{ D^2 \augHam(U_c) \frac{dU_c}{dc}, \frac{dU_c}{dc}},
        \end{equation}
        where the last equality follows from \eqref{D^2 E = DP relation}. 
         
        For each $\rho > 0$, let
        \[
            \tube_\rho^\Xspace \coloneqq \brac*{ u \in \nbhdO : \inf_{s \in \R}{\norm{ u - T(s) U_c }_{\Xspace}} < \rho }
        \] 
        be the tubular neighborhood of radius $\rho$ in $\Xspace$ for the $U_c$-orbit generated by $T$. We also define
        \[
            \mathcal{B}_R^\Wspace \coloneqq \brac*{ u \in \nbhdO \cap \Wspace : \inf_{s \in \R}{\norm{T(s)u}_{\Wspace}} < R }
        \]
        for all $R > 0$, which collapses to a ball if the symmetry group has no affine part.
        
        Our first result states that if $d''(c) > 0$ at a certain wave speed $c \in \cinterval$, then $U_c$ is conditionally orbitally \emph{stable}.
        
        \begin{theorem}[Stability]
            \label{abstract stability theorem}
            Suppose that the above assumptions hold.  If $d''(c) > 0$, then the bound state $U_c$ is conditionally orbitally stable in the following sense.  For any $R > 0$ and $\rho > 0$, there exists $\rho_0 > 0$ such that, if $u \colon [0,t_0) \to \mathcal{B}_R^\Wspace$ is a solution of \eqref{abstract Hamiltonian system}, with initial data $u_0 \in \tube_{\rho_0}^\Xspace$, then $u(t) \in \tube_{\rho}^\Xspace$ for all $t \in [0,t_0)$.
        \end{theorem}
        \begin{remark}
            \label{weaker hypotheses remark}
            As will become clear in the next section, the stability theorem holds under weaker hypotheses.  Most notably, we can drop the intersection with $\Dom(T'(0)|_\Wspace)$ in Assumption~\ref{abstract symmetry assumption}\ref{range density}.
        \end{remark}
        
        In order to prove an instability result, we need to know that \eqref{abstract Hamiltonian system} can be solved at least locally around the $U_c$-orbit. If we introduce
        \[
            \tube_\nu^\Wspace \coloneqq \brac*{ u \in \nbhdO \cap \Wspace : \inf_{s \in \R}{\norm{ u - T(s) U_c}_{\Wspace}} < \nu }
        \]
        for $\nu > 0$, we mean the following.
        
        \begin{assumption}[Local existence]
            \label{abstract LWP assumption}
            There exists $\nu_0 > 0$ and $t_0 > 0$ such that for all initial data $u_0 \in \tube_{\nu_0}^\Wspace$, there exists a unique solution to \eqref{abstract Hamiltonian system} on the interval $[0,t_0)$.
        \end{assumption}
        
        With the above hypothesis, we can conclude that if $d''(c) < 0$, then $U_c$ is conditionally orbitally \emph{unstable}.
        
        \begin{theorem}[Instability]
            \label{abstract instability theorem}
            If $d''(c) < 0$ and Assumption~\ref{abstract LWP assumption} is satisfied, then the bound state $U_c$ is orbitally unstable: There exists a $\nu_0 > 0$ such that for every $0 < \nu < \nu_0$, there exists initial data in $\tube_\nu^\Wspace$ whose corresponding solution exits $\tube_{\nu_0}^\Wspace$ in finite time.
        \end{theorem}
        
        If $\Xspace = \Wspace$, we also obtain a more conventional stability result as a corollary of Theorem~\ref{abstract stability theorem}.
        
        \begin{corollary}[Stability when $\Xspace = \Wspace$]
            \label{X = W stability corollary} 
            If $d''(c) > 0$, Assumption~\ref{abstract LWP assumption} holds, and $\Xspace = \Wspace$, then the bound state $U_c$ is orbitally stable: For any $\nu > 0$, there exists $\nu_0 > 0$ such that the solution for any initial data $u_0 \in \tube_{\nu_0}^\Wspace$ exists globally, and stays in $\tube_\nu^\Wspace$.
        \end{corollary}
        
        Together, Theorem~\ref{abstract instability theorem} and Corollary \ref{X = W stability corollary} essentially recover the classical GSS theory in the special case that $\Xspace = \Wspace$, $J$ is a state-independent isomorphism, $T(s)$ is linear, and the Hamiltonian system \eqref{abstract Hamiltonian system} is globally well-posed.  The only exception is that, in the interest of brevity, we have not addressed the situation where $d''(c) = 0$.  
        
        Lastly, let us comment on how the above results relate to the recent monumental paper of Lin and Zeng \cite{lin2017instability}, which studies the dynamics of linear Hamiltonian systems under weaker assumptions on the Poisson map than ours (for instance, they allow an infinite-dimensional kernel).   While this theory concerns the linear case, under some conditions it can be applied to construct invariant manifolds for nonlinear systems as well; see the work of Jin, Lin, and Zeng \cite{jin2018dynamics,jin2018invariant}.  When this can be accomplished, it gives considerably more information than the conditional orbital stability/instability we obtain from Theorem~\ref{abstract stability theorem} or Theorem~\ref{abstract instability theorem}.  However, the methodology has difficulty attacking equations for which the solution map incurs a loss of derivatives, such as quasilinear problems.  To overcome this, one needs the linear evolution to display sufficiently strong smoothing properties, which limits somewhat its applicability. By contrast, the framework we present here is adapted to the quasilinear setting by design, and does not rely on linear estimates.  

\section{Stability in the general setting}
    \label{abstract stability section}

    The purpose of this section is to prove Theorem~\ref{abstract stability theorem} on the conditional orbital stability of the bound state $U_c$, under the hypothesis that $d''(c) > 0$.  Our basic approach follows the ideas of Grillakis, Shatah, and Strauss, but many adaptations are required due to the more complicated functional analytic setting. Interestingly, the state dependence of $J$ is less of an issue for this argument than one may expect.
    
    We begin with a technical lemma which states that, in a sufficiently small tubular neighborhood $\tube_\rho^\Xspace$ of $U_c$, one can find a parameter value $s$ (depending on $u$) such that the distance between $T(s)u$ and $U_c$ in the energy norm is minimized.
    
    \begin{lemma}
        \label{mod out T lemma}
        If $s \mapsto T(s) U_c$ is not periodic, then then exists a $\rho > 0$ and a function $\tilde{s} \in C^2(\tube_{\rho}^\Xspace ; \R)$ such that, for all $u \in \tube_{\rho}^\Xspace$ the following holds.  
        \begin{enumerate}[label={\rm(\alph*)}]
            \item \label{tilde s minimizer}
                $ \norm{ T(\tilde{s}(u))u - U_c }_{\Xspace} \leq \norm{ T(r) u - U_c }_{\Xspace}$, for all $r \in \R$. 
            \item \label{tilde s orthogonal}
                $( T(\tilde{s}(u)) u - U_c, T'(0) U_c)_{\Xspace} = 0$.
            \item \label{composition tilde s T}
                $\tilde{s}(T(r)u) = \tilde{s}(u) - r$ for all $r \in \R$.
            \item \label{s tilde prime formula}
                For all $u \in \tube_\rho^\Xspace$ and $v \in \Xspace$, 
                \begin{align*}
                    \jb{D\tilde{s}(u),v} &= -\frac{\jb{\sigma_1(u),v}}{r_1(u)},\\
                    \jb{D^2\tilde{s}(u)v,v} &= -\frac{r_2(u)\jb{\sigma_1(u),v}^2}{r_1(u)^3} - 2 \frac{\jb{\sigma_1(u),v}\jb{\sigma_2(u),v}}{r_1(u)^2},
                \end{align*}
                where
                \begin{align*}
                    \sigma_k(u) &\coloneqq IT^{(k)}(-\tilde{s}(u))U_c, \qquad k = 1,2,\\
                    r_1(u) &\coloneqq \norm{T'(0)U_c}_{\Xspace}^2 - (T(\tilde{s}(u))u-U_c,T''(0)U_c)_{\Xspace},\\
                    r_2(u) &\coloneqq (T(\tilde{s}(u))u-U_c,T'''(0)U_c).
                \end{align*}
            \item \label{image tube under tilde s prime}
                We have $D\tilde{s}(u) \in \Dom(\hat{J})$ for every $u \in \tube_{\rho}^\Xspace$, and the map $g \colon \tube_{\rho}^\Xspace \cap \Wspace \to \Wspace$ defined by $g(u) \coloneqq J(u)D\tilde{s}(u)$ is of class $C^1(\tube_{\rho}^\Xspace \cap \Wspace ; \Wspace)$. 
        \end{enumerate}
        If instead $s \mapsto T(s) U_c$ has minimal period $\ell$, then the same result is true except $\tilde s \in C^2(\tube_\rho^\Xspace; \R/(\ell\R))$, and the equality in part~\ref{composition tilde s T} holds modulo $\ell$.
    \end{lemma}
    \begin{proof}
        For $s \in \R$ and $u \in \Xspace$, set 
        \[
            h(u,s) \coloneqq \frac{1}{2}\norm{ T(s) u - U_c }_{\Xspace}^2 = \frac{1}{2}\norm{u-T(-s)U_c}_{\Xspace}^2.
        \]
        Then
        \begin{align*}
            \partial_s h(u,s) &= (T(s)u-U_c,T'(0)U_c)_\Xspace,\\
            \partial_s^2 h(u,s) &=\norm{T'(0)U_c}_{\Xspace}^2 - (T(s)u-U_c,T''(0)U_c)_\Xspace.
        \end{align*}
        Clearly $\partial_s h(U_c,0) = 0$ and $\partial_s^2 h(U_c,0) = \norm{T'(0)U_c}_{\Xspace}^2 > 0$. The implicit function theorem then ensures the existence of a ball $B_\delta \subset \Xspace$ centered at $U_c$, an interval $(-s_0, s_0)$, and a $C^2$ map $\tilde s\colon B_\delta \to (-s_0, s_0)$ such that the equation $\partial_s h(u,s)=0$ has a unique solution $s=\tilde s(u)\in (-s_0, s_0)$ for all $u\in B_\delta$. Thus $s = \tilde s(u)$ uniquely minimizes $h(u,\placeholder)$ on $(-s_0,s_0)$ for any fixed $u \in B_\delta$.
        
        We will only present the argument for the non-periodic orbits as the proof for the periodic case requires only a simple modification.  Assumption~\ref{bound state assumption}\ref{non-periodic bound state assumption} then guarantees that there exists an $\eta > 0$ such that 
        \[
            \inf_{s \geq s_0}{\norm{T(s)U_c -U_c}_\Xspace} \geq \eta.
        \]
        Let $\rho \coloneqq \min(\eta/3,\delta)$. Then, if $u \in B_\rho$ and $r \in \R$ are such that $\norm{ T(r) u - U_c }_{\Xspace} \leq \norm{T(\tilde s(u)) u - U_c }$, we have 
        \begin{align*}
            \norm{T(r)U_c-U_c}_\Xspace &= \norm{dT(r)(U_c-u) + T(r)u-U_c}_{\Xspace} \leq \norm{U_c - u}_{\Xspace} + \norm{T(\tilde{s}(u))u-U_c}_\Xspace\\
            &\leq 2\norm{u - U_c}_{\Xspace} < \eta,
        \end{align*}
        which implies that $r \in (-s_0, s_0)$ and hence $r = \tilde s(u)$ by uniqueness.  This completes the proof of parts~\ref{tilde s minimizer} and \ref{tilde s orthogonal} for $u \in B_\rho$.
        
        For part~\ref{composition tilde s T}, note that if both $u$ and $T(r)u$ lie in $B_\rho$, then
        \[
            \norm{T(\tilde{s}(u)-r)T(r)u - U_c}_\Xspace = \norm{T(\tilde{s}(u))u - U_c}_\Xspace \leq \norm{T(t)u - U_c}_{\Xspace}
        \]
        for all $t \in \R$. In particular, if we choose $t = \tilde{s}(T(r)u)+r$, we obtain part~\ref{composition tilde s T} on $B_\rho$ by uniqueness. Moreover, as a consequence, we can proceed to extend $\tilde{s}$ to all of $\tube_\rho^\Xspace$ through
        \[
            \tilde{s}(u) = \tilde{s}(T(r)u) + r,
        \]
        where $r$ is such that $T(r)u \in B_\rho$. This is well defined, since if both $T(r)u$ and $T(s)u$ lie in $B_\rho$, then
        \[
            \tilde{s}(T(s)u) = \tilde{s}(T(s-r)T(r)u) = \tilde{s}(T(r)) - (s-r)
        \]
        by part~\ref{composition tilde s T} on $B_\rho$.
        
        The identities in part~\ref{s tilde prime formula} follow by straightforward calculations. Finally, for any $u \in \tube_{\rho}^\Xspace $, we have $\sigma_1(u) \in \Dom(\hat{J})$ by Assumption~\ref{bound state assumption}\ref{bound state domain assumption} and \eqref{abstract commutation identity}, and since moreover
        \[
            J(u)\sigma_1(u) = B(u)dT(-\tilde{s}(u))\hat{J}IT'(0)U_c
        \]
        by \eqref{abstract commutation identity}, part~\ref{image tube under tilde s prime} follows from \eqref{second technical bound state assumption} and \Cref{abstract symplectic assumption}\ref{regularity of J}.
    \end{proof}
    
    By Assumption~\ref{spectral assumptions}\ref{spectrum config assumption}, we know that $\Xspace$ admits the spectral decomposition 
    \[
        \Xspace = \Xspace_- \oplus \Xspace_0 \oplus \Xspace_+,
    \]
    where $\Xspace_- \coloneqq \linspan{\{\chi_c\}}$, $\Xspace_0 = \linspan{\{T^\prime(0) U_c\}}$, and $\Xspace_+$ is the positive subspace of $I^{-1}\Hc$.  Here we are using the fact that $T^\prime(0) U_c$ is a generator for the kernel of $I^{-1} \Hc$.  Observe that the restriction of $I^{-1} \Hc$ to $\Xspace_+$ is a positive operator, in the sense that there exists an $\alpha = \alpha(c) > 0$ such that 
    \begin{equation}
        \label{Hc positivity}
        \jb*{ \Hc v, v } \geq \alpha \norm{ v}_{\Xspace}^2 \qquad \text{for all } v \in \Xspace_+.
    \end{equation}
    
    The following lemma describes a version of this inequality which holds also outside $\Xspace_+$.
    \begin{lemma}
        \label{lower bound Hc lemma}
        Suppose that $y \in \Xspace$ is such that $\jb{\Hc y,y} < 0$.  Then there exists a constant $\tilde{\alpha} > 0$ such that
        \begin{equation}
            \label{lower bound Hc}
            \jb{\Hc v,v} \geq \tilde{\alpha} \norm{v}_{\Xspace}^2
        \end{equation}
        for every $v \in \Xspace$ satisfying
        \begin{equation}
            \label{lower bound Hc v assumption}
            \jb{\Hc y,v} = 0 \qquad \text{and} \qquad (T'(0)U_c,v)_\Xspace = 0.
        \end{equation}
    \end{lemma}
    \begin{proof}
        We decompose $y$ as
        \[
            y = a_0 \chi_c + b_0 T^\prime(0) U_c + p_0, \qquad \text{for some } a_0, b_0 \in \R, ~ p_0 \in \Xspace_+,
        \]
        from which we compute that
        \[
            \jb{\Hc y,y } = -a_0^2\mu_c^2 + \jb{\Hc p_0, p_0},
        \]
        or
        \begin{equation}
            \label{a0 expression}
            a_0^2\mu_c^2 = \jb{H_c p_0, p_0} + \abs{\jb{H_c y,y}},
        \end{equation}
        which in particular implies that $a_0^2 > 0$.
        
        Now, let $v$ be as in the statement of the lemma. Using the spectral decomposition of $\Xspace$, we may likewise write
        \[
            v = a \chi_c + p, \qquad \text{for some }  a \in \R, ~ p \in \Xspace_+,
        \]
        as $v$ has no component in $\Xspace_0$, by assumption. Moreover, we have
        \[
            0 = \jb{\Hc y,v} = - a_0 a \mu_c^2 + \jb{\Hc p_0, p},
        \]
        and therefore
        \begin{equation}
            \label{a expression}
            a = \frac{\jb{\Hc p_0, p}}{a_0 \mu_c^2}.
        \end{equation}
        It follows that
        \begin{align*}
            \jb{\Hc v, v} &= -a^2 \mu_c^2 + \jb{\Hc p,p} = - \frac{\jb{\Hc p_0, p}^2}{a_0^2\mu_c^2} + \jb{\Hc p,p}\\
            &\geq \parn*{1 - \frac{\jb{\Hc p_0, p_0}}{a_0^2 \mu_c^2}}\jb{\Hc p, p}= \frac{\abs{\jb{\Hc y,y}}}{a_0^2\mu_c^2} \jb{H_c p,p} \geq  \alpha\frac{\abs{\jb{\Hc y,y}}}{a_0^2\mu_c^2} \norm{p}_\Xspace^2\\
        \end{align*}
        by the Cauchy--Schwarz inequality applied to $\Hc|_{\Xspace_+}$, \eqref{a0 expression}, and \eqref{Hc positivity}. Finally, the result now follows by combining this inequality with
        \[
            \norm{v}_{\Xspace}^2 = a^2 + \norm{p}_{\Xspace}^2 \leq \parn*{\frac{\norm{\Hc p_0}_{\Xspace^*}^2}{a_0^2 \mu_c^4}+1} \norm{p}_{\Xspace}^2,
        \]
        where we have utilized \eqref{a expression}.
    \end{proof}
    
    We obtain the following as a corollary.
    
    \begin{corollary}
        \label{lower bound Hc corollary}
        Suppose that $d''(c) > 0$. Then there exists a constant $\tilde{\alpha} > 0$ such that \eqref{lower bound Hc} holds for every $v \in \Xspace$ satisfying
        \[
            \jb{\nabla\mom(U_c),v} = 0 \qquad \text{and} \qquad (T'(0)U_c,v)_\Xspace = 0.
        \]
    \end{corollary}
    \begin{proof}
        If $d''(c) > 0$, we may apply Lemma~\ref{lower bound Hc lemma} with $y = \dUdc$, by \eqref{abstract d'' identity}. Furthermore, we have $\Hc \dUdc = \nabla \mom(U_c)$ due to \eqref{D^2 E = DP relation}.
    \end{proof}
    
    Note that in the setting of Lemma~\ref{mod out T lemma}(a),
    \[
        \norm{ T(\tilde{s}(u)) - U_c }_{\Xspace}  = \inf_{r \in \R}{\norm{ T(r)u - U_c }_{\Xspace}} < \rho, \qquad \text{for all } u \in \tube_\rho^\Xspace,
    \]
    whence it makes sense to define the map
    \[
        M \colon \tube_\rho^\Xspace \ni u   \mapsto T(\tilde{s}(u)) u \in \tube_\rho^\Xspace
    \] 
    whenever $\rho > 0$ is small enough for the lemma to apply. Note that $M$ is also invariant under the action of $T$, as
    \begin{equation}
        \label{M invariant under T identity}
        M(T(s)u) = T(\tilde{s}(T(s)u))T(s)u = T(\tilde{s}(u) - s)T(s) u = T(\tilde{s}(u)) u = M(u),
    \end{equation}
    where the second equality comes from Lemma~\ref{mod out T lemma}(c). Moreover, we are able to bound $M(u)$ in the smoother norm.
    
    \begin{lemma}
        \label{upper bound M lemma}
        Let $R > 0$, and suppose that $\rho > 0$ is like in Lemma~\ref{mod out T lemma}. Then
        \[
            \norm{M(u)}_{\Wspace} \leq R + \omega(\rho  + \norm{\iota_{\Wspace \hookrightarrow \Xspace}} R + \norm{U_c}_{\Xspace}) \qquad \text{for all } u \in \tube_\rho^\Xspace \cap \mathcal{B}_R^\Wspace.
        \]
    \end{lemma}
    \begin{proof}
        If $u \in \tube_{\rho}^\Xspace \cap \mathcal{B}_R^\Wspace$, then in particular there exists an $r \in \R$ such that $\norm{T(r)u}_{\Wspace} < R$. Set $v = T(r)u$, and observe that
        \[
            \norm{M(u)}_{\Wspace} = \norm{M(v)}_{\Wspace} = \norm{dT(\tilde{s}(v))v + T(\tilde{s}(v))0}_{\Wspace} \leq R + \omega(\norm{T(\tilde{s}(v))0}_{\Xspace})
        \]
        by \eqref{M invariant under T identity} and Assumption~\ref{abstract symmetry assumption}\ref{affine bound assumption}. The final bound is obtained by combining this inequality with
        \begin{align*}
            \norm{T(\tilde{s}(v))0}_{\Xspace} &= \norm{M(v) - U_c + U_c - dT(\tilde{s}(v))v}_{\Xspace}\\
            &\leq \rho + \norm{U_c}_{\Xspace} + \norm{\iota_{\Wspace \hookrightarrow \Xspace}}R,
        \end{align*}
        where we have used that $v \in \tube_\rho^\Xspace$, since $\tube_\rho^\Xspace$ is invariant under $T$, and that $\Wspace$ embeds continuously into $\Xspace$.
    \end{proof}
    
    We will now use Lemmas \ref{lower bound Hc lemma} and \ref{upper bound M lemma} to obtain the key inequality needed to prove stability. It is convenient to introduce the notation
    \begin{equation}
        \label{P manifold}
        \Pmanifold_c \coloneqq \brac*{ u \in \nbhdO \cap \Vspace  : \mom(u) = \mom(U_c) }
    \end{equation}
    for the level set of the momentum associated with $U_c$.
    
    \begin{lemma}
        \label{lower bound energy difference lemma}
        Suppose that $d''(c) > 0$. Then, for any $R > 0$, there exist $\rho > 0$ and $\beta > 0$ such that
        \begin{equation}
            \label{key stability lower bound}
            E(u) - E(U_c) \geq \beta \norm{M(u)-U_c}_{\Xspace}^2 \qquad \text{for all } u \in \tube_{\rho}^\Xspace \cap \Pmanifold_c \cap \mathcal{B}_R^\Wspace.
        \end{equation}
        Moreover, the assumption that $d''(c) > 0$ can be removed under the additional restriction that $\jb{\Hc y,M(u)-U_c} = 0$ for a fixed $y \in \Xspace$ such that $\jb{\Hc y,y} < 0$.
    \end{lemma}
    \begin{proof}
        Let $u$ be as in the statement of the lemma, and set $v \coloneqq M(u) - U_c$. Expanding $E_c$ in a neighborhood of $U_c$ in $\Vspace$, recalling that $U_c$ is a critical point and that both the energy and momentum are conserved by the group, yields
        \begin{equation}
            \label{augHam expansion} 
            E_c(u) = E_c(U_c + v) = E_c(U_c) + \frac{1}{2}\jb{H_c v,v} + O(\norm{v}_{\Vspace}^3).
        \end{equation}
        Note that $(v,T'(0)U_c)_{\Xspace} = 0$ by Lemma~\ref{mod out T lemma}\ref{tilde s orthogonal}, so if in addition $\jb{\Hc y,v} = 0$, then Lemma~\ref{lower bound Hc lemma} ensures the existence of an $\tilde{\alpha} > 0$, independent of $v$, such that
        \[
            \jb{H_c v,v} \geq \tilde{\alpha} \norm{v}_{\Xspace}^2.
        \]
        
        If, on the other hand, $d''(c) > 0$, we decompose $v$ as
        \begin{equation}
            \label{ortho decomposition v}
            v = \lambda N + w, \qquad N \coloneqq I^{-1}\nabla P(U_c),
        \end{equation}
        with $(N,w)_{\Xspace}= 0$. Taking the inner product of both sides of \eqref{ortho decomposition v} with $N$, and using that $P(U_c + v) = P(U_c)$, we find
        \[
            \lambda \norm{N}_{\Xspace}^2 = (v,N)_{\Xspace} = \jb{DP(U_c),v} = O(\norm{v}_{\Vspace}^2),
        \]
        whence $\lambda = O(\norm{v}_{\Vspace}^2)$. It follows that
        \[
            \jb{H_c v,v} = \jb{H_c w,w} + O(\norm{v}_{\Vspace}^3).
        \]
        
        We wish to apply Corollary~\ref{lower bound Hc corollary} to obtain a lower bound for $\jb{H_c w,w}$. In that connection, observe that $\jb{\nabla P(U_c),w} = 0$, as $w$ is orthogonal to $N$ by construction. Moreover,
        \[
            (w,T'(0)U_c)_{\Xspace} = (v,T'(0)U_c) - \lambda \jb{\nabla P(U_c),T'(0)U_c} = 0
        \]
        in view of Lemma \ref{mod out T lemma} and \eqref{abstract T'(0) and P' identity}. Thus
        \[
            \jb{H_c v,v} \geq \tilde{\alpha}\norm{w}_{\Xspace}^2 + O(\norm{v}_{\Vspace}^3),
        \]
        where we can eliminate $w$ in favor of $v$ by observing that
        \[
            \norm{w}_{\Xspace}^2 \geq \parn*{\norm{v}_{\Xspace} - \abs{\lambda} \norm{N}_{\Xspace}}^2 \geq \norm{v}_{\Xspace}^2 - O(\norm{v}_{\Vspace}^3).
        \]
        
        In either case, the desired lower bound \eqref{key stability lower bound} follows if we can control the cubic $O(\norm{v}_{\Vspace}^3)$-remainder in \eqref{augHam expansion} by using the quadratic $\norm{v}_{\Xspace}^2$. This is precisely the motivation behind Assumption~\ref{abstract interpolation assumption}.  Indeed, \eqref{abstract interpolation inequality} and Lemma~\ref{upper bound M lemma} imply that 
        \[
            \norm{v}_{\Vspace}^3 \leq C \norm{v}_{\Xspace}^{2+\theta} \norm{v}_{\Wspace}^{1-\theta} \leq C\rho^\theta [R + \omega(\rho  + \norm{\iota_{\Wspace \hookrightarrow \Xspace}} R + \norm{U_c}_{\Xspace}) + \norm{U_c}_{\Wspace}]^{1-\theta} \norm{ v }_{\Xspace}^2,
        \]
        which enables us to absorb the remainder into the quadratic term by taking sufficiently small $\rho$. Note that we can replace $E_c$ by $E$ due to the assumption that $u \in \Pmanifold_c$.
    \end{proof}
    
    We are now prepared to prove the main theorem of the section on the conditional orbital stability of the bound state $U_c$.
    
    \begin{proof}[Proof of Theorem~\ref{abstract stability theorem}]
        Seeking a contradiction, suppose there exist $R > 0$, $\rho > 0$, and a sequence of solutions $u_n \colon [0, t_0^n) \to \mathcal{B}_R^\Wspace$, with initial data $u_0^n$, such that $\norm{M(u_0^n)-U_c}_{\Xspace} \to 0$, but for which
        \[
            \norm{M(u_n(\tau_n))-U_c}_{\Xspace} = \rho
        \]
        for some $\tau_n \in (0, t_0^n)$. Without loss of generality, we may take $\tau_n$ to be the \emph{first} time that $u_n$ exits $\tube_{\rho}^\Xspace$. Moreover, we can shrink $\tube_{\rho}^\Xspace$ such that Lemma~\ref{lower bound energy difference lemma} applies. Together with the conservation of energy and momentum, we deduce the existence of a $\beta > 0$ such that
        \[
            E(u_0^n) - E(U_c) \geq \beta \norm{M(u_n(\tau_n))-U_c}_{\Xspace}^2 = \beta \rho ^2
        \]
        for every $n$. On the other hand, $E(u_0^n) = E(M(u_0^n))$, and $\norm{M(u_0^n) - U_c}_{\Xspace} \to 0$. Combined with the fact that $\sup_{n}{\norm{M(u_0^n) - U_c}_{\Wspace}} \lesssim_R 1$ by Lemma~\ref{upper bound M lemma}, we can use Assumption~\ref{abstract interpolation assumption} to deduce that $M(u_0^n) \to  U_c$ in $\Vspace$, and therefore that $E(u_0^n) \to E(U_c)$. But this contradicts the strictly positive lower bound on $E(u_0^n) - E(U_c)$ derived above, and hence we have arrived at a contradiction.  
    \end{proof}

\section{Instability in the general setting}
    \label{abstract instability}
    
    This section is devoted to proving Theorem~\ref{abstract instability theorem} on the conditional orbital instability of $U_c$, under the hypothesis that the moment of instability satisfies $d''(c) < 0$.  In contrast to Section~\ref{abstract stability section}, the state-dependence of the Poisson map $J$ presents a more serious technical challenge to the analysis here.
    
    \subsection{Identification of a negative direction}
        Because we do not assume that $J(u)$ is surjective, and because $\chi_c$ does not necessarily lie in $\Wspace$, we must make further modifications to the GSS program.  The next lemma shows that it is possible to find a negative direction $z \in \Wspace$ that is not only tangent to $\Pmanifold_c$, but also lies in the range of a restriction of $J(U_c)$.  This follows from a surprisingly simple density argument.  
        
        \begin{lemma} \label{existence of z lemma}
            Suppose that $d''(c) < 0$. Then there exists $z \in \Dom(T'(0)|_\Wspace)$, of the form $z = J(U_c)IZ$ for some $Z \in \Dom(T'(0))$, such that 
            \begin{equation}
                \label{properties of z}
                \jb[\big]{ D^2 \augHam(U_c) z, z } < 0 \qquad \text{and} \qquad \jb*{ D\mom(U_c), z } = 0. 
            \end{equation}
        \end{lemma}
        \begin{proof}
            For ease of notation, we once again set $N \coloneqq I^{-1}\nabla P (U_c)$. Defining the quadratic form $Q \in C^0(\Xspace; \R)$ by 
            \[
                Q(u) \coloneqq \jb*{ \Hc u, u},
            \]
            we see that \eqref{properties of z} can be rephrased as $z$ satisfying
            \[
                Q(z) < 0 \qquad \text{and} \qquad \parn*{ N, z}_{\Xspace} = 0.
            \]
            
            The element
            \[
                y \coloneqq \frac{\jb{\nabla  P(U_c),\chi_c}}{d''(c)} \dUdc + \chi_c \in \Xspace
            \]
            satisfies both of these properties because
            \[
                Q(y) = \frac{\jb{\nabla  P(U_c),\chi_c}^2}{d''(c)} - \mu_c^2 < 0, \qquad (N,\,y)_{\Xspace} = 0
            \]
            by \eqref{D^2 E = DP relation} and \eqref{abstract d'' identity}. However, $y$ does not necessarily lie in $\Dom(T'(0)|_\Wspace)$, nor must it be in the range of $J(U_c)$.
            
            Note that $B(U_c)$ restricts to an isomorphism on $\Dom(T'(0))|_\Wspace$ by \eqref{abstract commutation identity}. Thus, if $J(U_c)IZ \in \Dom(T'(0)|_\Wspace)$, then $\hat{J}IZ \in \Dom(T'(0)|_\Wspace)$, and consequently $Z \in \Dom(T'(0))$ by \eqref{abstract derivative commutation identity}. To complete the proof, it suffices to show that
            \[
                \mathbb{U} \coloneqq \Dom(T'(0)|_\Wspace) \cap \Rng{J(U_c)}
            \]
            is dense in $N^\perp$, where $N^\perp \coloneqq \brac{ u \in \Xspace : (N, u)_{\Xspace} = 0}$. Recall that $\mathbb{U}$ is dense in $\Xspace$ due to Assumption~\ref{abstract symmetry assumption}\ref{range density}.
            
            First we claim that there exists $v \in \mathbb{U}$ such that $(N, v)_{\Xspace} \neq 0$.  Were this not the case, we would have $\mathbb{U} \subset N^\perp$, which would contradict its density in $\Xspace$.  Without loss of generality, we may choose $v$ such that $(N, v)_{\Xspace} = 1$. Now, let $u \in N^\perp$ be given.  By density, there exists an approximating sequence $\{ u_n \} \subset \mathbb{U}$ with $u_n \to u$ in $\Xspace$.  Putting
            \[
                w_n \coloneqq u_n - \parn*{N, u_n}_{\Xspace} v,
            \]
            we see that the sequence $\brac{ w_n } \subset N^\perp \cap \mathbb{U}$, and that
            \[
                w_n \to u - \parn*{N, u}_{\Xspace} v = u \text{ in } \Xspace,
            \]
            whence $\mathbb{U}$ is indeed dense in $N^\perp$.  
            
            By the argument above, there is a sequence $\brac{ z_n } \subset \mathbb{U} \cap N^\perp$ such that $z_n \to y$ in $\Xspace$. For $n$ sufficiently large, $Q(z_n) < 0$ by continuity, and so the lemma is proved.
        \end{proof}

    \subsection{Lyapunov function}  
        
        In the previous subsection, we constructed a vector $z$ in the negative cone of $\Hc$, that moreover is tangent to the fixed momentum manifold $\Pmanifold_c$ at $U_c$.  The strategy at this point is to use $z$ to build a Lyapunov function for the abstract Hamiltonian system \eqref{abstract Hamiltonian system}, and thereby prove instability.  
         
        In the next lemma, we follow Grillakis, Shatah, and Strauss by introducing a functional $A$; designed so that the corresponding Hamiltonian vector field (i) points in the direction $z$ at $U_c$, and (ii) is in the kernel of $D\mom$ in a tubular neighborhood of $U_c$.  
        
        \begin{lemma}
            \label{existence of A lemma}
            There exists a $\rho > 0$ and a functional $A \in C^1(\tube_\rho^\Xspace; \R)$ having the following properties:
            \begin{enumerate}[label=\rm(\alph*)]
                \item \label{A is invariant under T}
                    $A(T(s)u) = A(u)$, for all $u \in \tube_\rho^\Xspace$ and $s \in \R$.
                \item \label{DA in domain of J}
                    $DA(u) \in \Dom(\hat{J})$, for all $u \in \tube_\rho^\Xspace$.
                \item \label{JDA at U_c}
                    $J(U_c) DA(U_c) = -z$, where $z$ is like in Lemma~\ref{existence of z lemma}.
                \item \label{JDA regularity}
                    The mapping $u \mapsto J(u)DA(u)$ is of class $C^1(\tube_{\nu}^\Wspace ; \Wspace)$, where $\nu > 0$ is such that $u \in \tube_\nu^\Wspace \subset \tube_{\rho}^\Xspace$; and
                \item \label{JDA is tangent to manifold}
                    $\jb{D \mom(u),J(u)DA(u)} = 0$ for all $u \in \tube_\nu^\Wspace$.
            \end{enumerate}
        \end{lemma}
        \begin{proof}
            Let $z$ and $Z$ be given as in Lemma~\ref{existence of z lemma}, and choose $\rho > 0$ so that Lemma \ref{mod out T lemma} applies. Put
            \begin{equation}
                \label{def A(u)} 
                A(u) \coloneqq - (Z,M(u)-U_c)_\Xspace \qquad \text{for all } u \in \tube_\rho^\Xspace,
            \end{equation}
            for which part~\ref{A is invariant under T} follows immediately from the corresponding property of $M$ established in \eqref{M invariant under T identity}. The regularity of $\tilde{s}$, and the properties of $Z$, also show that $A$ is $C^1$ with
            \begin{equation}
                \label{formula A'(u) before simplification}
                DA(u) = ((dT'(-\tilde{s}(u))Z,u)_{\Xspace} - (Z,T'(\tilde{s}(u))0)_{\Xspace})D\tilde{s}(u) -IdT(-\tilde{s}(u))Z
            \end{equation}
            for all $u \in \tube_\rho^\Xspace$. Since $D\tilde{s}(u)$ lies in $\Dom(\hat{J})$ by Lemma \ref{mod out T lemma}, while $IdT(-\tilde{s}(u))Z$ is in $\Dom(\hat{J})$ by Assumption~\ref{abstract symmetry assumption}\ref{invariances}, this proves part~\ref{DA in domain of J}.
            
            Next, choose $\nu > 0$ such that $\tube_{\nu}^\Wspace \subset \tube_{\rho}^\Xspace$. When  $u \in \tube_\nu^\Wspace \cap \Dom(T'(0)|_\Wspace)$, the formula for $DA(u)$ in \eqref{formula A'(u) before simplification} simplifies to
            \begin{equation}
                \label{formula A'(u)}
                DA(u) = \jb{D\mom(u),h(u)}D\tilde{s}(u) - IdT(-\tilde{s}(u))Z,
            \end{equation}
            with
            \begin{equation}
                \label{formula h}
                h(u)\coloneqq J(u)IdT(-\tilde{s}(u))Z = B(u)dT(-\tilde{s}(u))B(U_c)^{-1}z.
            \end{equation}
            Here we have used \eqref{abstract T'(0) and P' identity}, Assumption~\ref{abstract symmetry assumption}\ref{commutativity assumption}, and the skew-adjointness of $J(u)$. By density of $D(T'(0)|_\Wspace)$ in $\Wspace$, the formula in \eqref{formula A'(u)} is, in fact, valid for every $u \in \tube_\nu^\Wspace$.
            
            Moreover, applying $J(u)$ to \eqref{formula A'(u)} leads to the expression
            \begin{equation}
                \label{J(u)DA(u) formula}
                J(u)DA(u) = \jb{D\mom(u),h(u)}g(u) - h(u),
            \end{equation}
            where $g$ is the function defined in Lemma~\ref{mod out T lemma}. We have already confirmed that $g$ has the required properties for part~\ref{JDA regularity}, and in light of \eqref{formula h} and the fact that $z \in \Dom(T'(0)|_\Wspace)$ and \eqref{abstract commutation identity}, so does $u \mapsto J(u)DA(u)$. From Lemma~\ref{mod out T lemma} we see that $\tilde{s}(U_c) = 0$, and therefore $h(U_c) = z$. Evaluating \eqref{J(u)DA(u) formula} at $u = U_c$ then yields
            \begin{equation}
                \label{formula A'(U_c) 1}
                J(U_c)DA(U_c) = \jb{D \mom(U_c),z}g(U_c) - z = -z, 
            \end{equation}
            by \eqref{properties of z}, which is part~\ref{JDA at U_c}.
            
            Finally, since the map $s \mapsto A(T(s)u)$ has derivative
            \[
                0 = \jb{DA(u),T'(0)u} = \jb{DA(u),J(u)\nabla P (u)} = - \jb{DP(u),J(u)DA(u)}
            \]
            at $s = 0$ for every $u \in \tube_\nu^\Wspace \cap \Dom(T'(0)|_\Wspace)$ by part~\ref{A is invariant under T}, part~\ref{JDA is tangent to manifold} follows after another appeal to density. Here we have once again made use of the identity \eqref{abstract T'(0) and P' identity}.
        \end{proof}
        
        With the functional $A$ in hand, we next consider the ordinary differential equation
        \begin{equation}
            \label{R ODE}
            \left\{ 
            \begin{alignedat}{2}
                &\dfrac{du}{d\lambda}  &&= -J(u(\lambda)) DA(u(\lambda)) \\
                & u(0)  && = v,
            \end{alignedat}
            \right.
        \end{equation}
        posed in $\tube_{\nu}^\Wspace$, where $\nu > 0$ is taken small enough for Lemma~\ref{existence of A lemma} to apply. Part~\ref{JDA regularity} of the lemma guarantees the existence of a unique solution, $\Phi = \Phi(\lambda, v) \in C^1(\mathcal{N};\tube_\nu^\Wspace)$, to \eqref{R ODE}, where
        \begin{equation}
            \label{flow domain}
            \mathcal{N} = \{(\lambda,v) \in \R \times \tube_{\nu_0}^\Wspace  : \abs{\lambda} < \lambda_0\},
        \end{equation}
        with $0 < \nu_0 < \nu$ and $\lambda_0 = \lambda_0(\nu_0) > 0$. By appealing to the commutation identities in \eqref{abstract commutation identity} and Lemma \ref{existence of A lemma}\ref{A is invariant under T}, we find 
        \begin{equation}
            \label{R commute with T identity}
            T(s) \Phi(\lambda, v) = \Phi(\lambda, T(s)v)
        \end{equation}
        whenever both sides of this equation make sense, which in particular justifies that $\lambda_0$ can be taken to be a constant in \eqref{flow domain}.
        
        Observe that
        \begin{equation}
            \label{Phi derivative at U_c}
            \partial_\lambda\Phi(0,U_c) = z
        \end{equation}
        as a result of Lemma~\ref{existence of A lemma}(c). Furthermore, since
        \[
            \frac{\partial}{\partial\lambda} \mom(\Phi(\lambda, v)) =- \jb*{ D \mom(\Phi(\lambda, v)), J(\Phi(\lambda, v)) DA(\Phi(\lambda, v)) } = 0
        \]
        by Lemma~\ref{existence of A lemma}(d), we have
        \begin{equation}
            \label{R ODE P conservation}
            \mom(\Phi(\lambda, v)) = \mom(v), \qquad \text{for all } (\lambda, v) \in \mathcal{N}.
        \end{equation}
        That is, the flow of \eqref{R ODE} preserves the momentum.
        
        \begin{lemma}[Lyapunov function]
            \label{existence of Lambda lemma}
            There exists a $\nu > 0$ and a functional $\Lambda \in C^1(\tube_\nu^\Wspace; \, \R)$, vanishing on the $U_c$-orbit, such that
            \[
                \eng(\Phi(\Lambda(v), v)) \geq \eng(U_c) \qquad \text{for all } v \in \tube_\nu^\Wspace \cap \Pmanifold_c.
            \]
        \end{lemma}
        One can interpret this lemma as follows.  Because of \eqref{R ODE P conservation}, the flow of \eqref{R ODE} leaves the momentum invariant but it may change the energy in either direction near $U_c$. By avoiding the problematic negative direction in a suitable way, and using Lemma~\ref{mod out T lemma}\ref{tilde s orthogonal} to deal with $\Xspace_0$ and the orbit under $T$, we can make sure that the energy does not decrease. 
        
        \begin{proof}[Proof of Lemma~\ref{existence of Lambda lemma}]  
        
        We wish to apply Lemma~\ref{lower bound energy difference lemma}. To that end, define the function $f \colon \mathcal{N} \to \R$ by 
        \begin{align*}
            f(\lambda, v) &\coloneqq \jb{\Hc z, M(\Phi(\lambda,v)) - U_c}\\
            &= \jb{\Hc z,dT(\tilde{s}(\Phi))(\Phi-U_c)} + \jb{\Hc z, T(\tilde{s}(\Phi))U_c - U_c},
        \end{align*}
        which evidently satisfies $f(0, U_c) = 0$. It is not obvious that this function is differentiable, but by differentiating the identity $\augHam(u) = \augHam(T(s)u)$, one finds that
        \[
            \jb{D^2 \augHam(T(-s)u)dT(-s)v,w} = \jb{D^2\augHam(u)v,dT(s)w}
        \]
        for all $s \in \R, u \in \nbhdO\cap\Vspace$, and $v,w \in \Vspace$, and thus in particular that
        \[
            f(\lambda,v) = \jb{D^2\augHam(T(-\tilde{s}(\Phi))U_c)dT(-\tilde{s}(\Phi))z,\Phi-U_c} + \jb{\Hc z,T(\tilde{s}(\Phi))U_c-U_c},
        \]
        holds for all $(\lambda, v) \in \mathcal{N}$. This expression shows that $f \in C^1(\mathcal{N};\R)$, as $\augHam \in C^3(\nbhdO \cap \Vspace;\R)$ and both $U_c$ and $z$ are in $\Dom(T'(0)|_\Wspace)$. Moreover, 
        \begin{align*}
            \partial_\lambda f(0,U_c) &= \jb{\Hc z, z} + \jb{D\tilde{s}(U_c),z}\jb{\Hc z, T'(0)U_c} = \jb{\Hc z,z} < 0,
        \end{align*}
        since $\jb{\Hc z,T'(0)U_c} = \jb{z,\Hc T'(0)U_c} = 0$.
        
        An application of the implicit function theorems tells us that there exists a neighborhood $\mathcal{V}$ of $U_c$ in $\tube_{\nu_0}^\Wspace$, and a $C^1$-mapping  $\Lambda \colon \mathcal{V} \to (-\lambda_0, \lambda_0)$ satisfying
        \begin{equation}
            \label{Lambda equation}
            f(\Lambda(v), v) = \jb{\Hc z, M(\Phi(\Lambda(v),v))-U_c} = 0\qquad \text{for all } v \in \mathcal{V}.
        \end{equation}
        In view of \eqref{R commute with T identity} and \eqref{M invariant under T identity}, we have $f(\lambda, T(s) v) = f(\lambda, v)$ for all $s \in \R$ and $(\lambda,v) \in \mathcal{N}$, so $\Lambda$ can be extended to a tubular neighborhood $\tube_\nu^\Wspace$ of the $U_c$-orbit.
        
        We may now use Lemma~\ref{lower bound energy difference lemma} to conclude that, possibly upon shrinking $\tube_\nu^\Wspace$, there exists a $\beta > 0$ such that
        \[
            \eng(\Phi(\Lambda(v),v)) - \eng(U_c) \geq \beta \norm{M(\Phi(\Lambda(v),v))-U_c}_{\Xspace}^2
        \]
        for every $v \in \tube_\nu^\Wspace \cap \Pmanifold_c$. In particular, the result follows.
        \end{proof}
        
        \begin{lemma}
            \label{energy inequality lemma}
            There exists a $\nu > 0$ such that
            \[
                \eng(U_c) \leq \eng(v) + \Lambda(v) \mathcal{S}(v) \qquad \text{for all } v \in \tube_\nu^\Wspace \cap \Pmanifold_c,
            \]
            wherein
            \begin{equation}
                \label{def S}
                \mathcal{S}(v) \coloneqq -\jb*{ D\eng(v), J(v) DA(v)}.
            \end{equation}
        \end{lemma}
        \begin{proof}
            Define the $C^1(\mathcal{N};\R)$-function 
            \[
                g(\lambda,v) \coloneqq \eng(\Phi(\lambda,v)) = \augHam(\Phi(\lambda,v))+cP(v),
            \]
            where we have exploited \eqref{R ODE P conservation} in evaluating $\mom$ at $v$. Then, suppressing the dependence of $\Phi$ on $(\lambda, v)$, we find 
            \[
                \partial_\lambda g(\lambda,v) = \jb{D\augHam(\Phi),\partial_\lambda \Phi} =  -\jb{D\augHam(\Phi),J(\Phi)DA(\Phi)},
            \]
            so $\partial_\lambda g \in C^1(\mathcal{N};\R)$ as well by Lemma~\ref{existence of A lemma}(e). It therefore makes sense to compute
            \[
                \partial_{\lambda}^2 g(0,U_c) = \jb{D^2 \augHam(U_c)\partial_{\lambda}\Phi(0,U_c), \partial_\lambda\Phi(0,U_c)} = \jb{H_c z,z} < 0,
            \]
            where we have used that $U_c$ is a critical point of $\augHam$, and the last inequality is Lemma~\ref{existence of z lemma}.
            
            We also see that
            \[
                \partial_\lambda g(0,v) = -\jb{D\augHam(v),J(v)DA(v)}=-\jb{D\eng(v),J(v)DA(v)} = \mathcal{S}(v)
             \] 
            for every $v \in \tube_{\nu_0}^\Wspace.$  It follows that
            \[
                g(\lambda,v) \leq g(0,v) + \partial_\lambda g(0,v) \lambda
            \]
            for small enough $\lambda$, and a possibly smaller neighborhood of $U_c$. This neighborhood can be made tubular, by the same reasoning as in the proof of Lemma~\ref{existence of Lambda lemma}. The desired upper bound now follows by setting $\lambda = \Lambda(v)$ and using Lemma~\ref{existence of Lambda lemma}.
        \end{proof}
        
        The final lemma we need in order to prove the instability theorem is the following.
        
        \begin{lemma}
            \label{initial data curve lemma}
            Suppose that $d''(c) < 0$. Then there exists a $C^2$-curve $\psi \colon (-1, 1) \to \Wspace$ such that
            \begin{enumerate}[label=\rm(\roman*)]
                \item $\psi(0) = U_c$ and $\psi^\prime(0) = z$; 
                \item   \label{curve on manifold}
                    $\psi(s) \in \Pmanifold_c$ for all $s \in (-1,1)$;
                \item $\eng \circ \psi$ has a strict local maximum at $0$.
            \end{enumerate}
        \end{lemma}
        \begin{proof}
            Define $\psi \colon (-\lambda_0,\lambda_0) \to \Wspace$ by $\psi(s) \coloneqq \Phi(s,U_c)$. Then $\psi(0) = U_c$ by definition of $\Phi$, while $\psi'(0) = z$ is \eqref{Phi derivative at U_c}. We also know that the flow of \eqref{R ODE} conserves momentum, whence $\psi(s) \in \Pmanifold_c$ for all $s \in (-\lambda_0,\lambda_0)$. Finally, the proof of Lemma~\ref{energy inequality lemma} shows that $\psi$ is $C^2$, and that $\eng \circ \psi$ has a strict local maximum at $0$. The result is now obtained by a possible reparameterization.
        \end{proof}
        \begin{remark}
            The properties of the curve in Lemma~\ref{initial data curve lemma} show that $U_c$ is \emph{not} a local minimizer of the constrained minimization problem
            \[
                \min{\brac*{ \eng(u) : u \in \Pmanifold_c}}.
            \]
        \end{remark}
        
    \subsection{Proof of the instability theorem}
        
        \begin{proof}
            Assume, to the contrary, that we do \emph{not} have instability. Then for every $\nu_0 > 0$, small enough for local existence from Assumption~\ref{abstract LWP assumption}, there exists a $0 < \nu < \nu_0$ such that solutions corresponding to initial data in $\tube_\nu^\Wspace$ exist globally in time and stay inside $\tube_{\nu_0}^\Wspace$. Fix such a $\nu_0$, which we also require to satisfy the hypotheses of the lemmas in this section.
            
            By the above reasoning, there exists a unique global in time solution
            \[
                u^s \in C^0([0,\infty),\tube_{\nu_0}^\Wspace)
            \]
            to \eqref{abstract Hamiltonian system}, with initial data $u^s(0) = \psi(s)$, for all $\abs{s}\ll 1$. Here, $\psi$ is the curve from Lemma~\ref{initial data curve lemma}. Since $\psi(s) \in \Pmanifold_c$ by Lemma~\ref{initial data curve lemma}\ref{curve on manifold}, the solutions $u^s$ all live on $\Pmanifold_c$ due to conservation of momentum.
            
            From Lemma~\ref{energy inequality lemma}, and conservation of energy, we obtain the inequality
            \[
                \eng(U_c) - \eng(\psi(s)) \leq \Lambda(u^s(t))\mathcal{S}(u^s(t))
            \]
            for all $\abs{s} \ll 1$ and $t \in [0,\infty)$. By choosing $\lambda_0 \leq 1$ in \eqref{flow domain}, we can assume that $\abs{\Lambda(u)} \leq 1$ for all $u \in \tube_{\nu_0}^\Wspace$. Thus
            \begin{equation}
                \label{S circ u^s lower bound}
                \abs{\mathcal{S}(u^s(t))} \geq \eng(U_c) - \eng(\psi(s)) > 0
            \end{equation}
            for all $0 < \abs{s} \ll 1 $ and $t \in [0,\infty)$, where the strict inequality stems from Lemma~\ref{initial data curve lemma}(iii). Moreover, by continuity, this implies that $\mathcal{S} \circ u^s$ does not change sign.
            
            Since $\hat{J} \colon \Dom(\hat{J}) \subset \Xspace^* \to \Xspace$ is a closed operator, we may view $\mathbb{D} \coloneqq \Dom(\hat{J})$ as a Banach space with the graph norm
            \[
                \norm{v}_{\mathbb{D}} \coloneqq \norm{v}_{\Xspace^*} + \norm{\hat{J}v}_\Xspace,
            \]
            and in this norm the map $u \mapsto J(u)$ is of class $C^0(\nbhdO \cap \Wspace; \Lin(\mathbb{D},\Xspace))$ by Assumption~\ref{abstract symplectic assumption}. It follows that the map $u \mapsto J(u)^*$ is in $C^0(\nbhdO \cap \Wspace; \Lin(\Xspace^*,\mathbb{D}^*))$. From \eqref{weak abstract Hamiltonian system} we obtain
            \[
                \frac{d}{dt} \jb{u^s(t),v} = -\jb{\nabla \eng(u^s(t)),J(u^s(t))v} = -\jb{J(u^s(t))^* \nabla E(u^s(t)),v}
            \]
            for every $v \in \mathbb{D}$. Now, since the embedding $\mathbb{D} \hookrightarrow \Xspace^*$ is dense, we have $\Xspace \hookrightarrow \mathbb{D}^*$. We can therefore view $u^s$ as a member of $C^0([0,\infty),\tube_{\nu_0}^\Wspace) \cap C^1((0,\infty),\mathbb{D}^*)$, with
            \[
                (u^s)'(t) = -J(u^s(t))^* \nabla E(u^s(t))
            \]
            for all $t \in (0,\infty)$. Furthermore, the functional $A$ in Lemma~\ref{existence of A lemma} can be viewed as a member of $C^1(\tube_{\nu_0}^\Wspace;\R)$, with the derivative $DA$ in $C^0(\tube_{\nu_0}^\Wspace; \mathbb{D})$. As the embedding $\Wspace \hookrightarrow \Xspace$ is dense and $\Wspace$ is reflexive, the embeddings $\Xspace^* \hookrightarrow \Wspace^*$ and $\mathbb{D} \hookrightarrow \Wspace^*$ are likewise dense.
            
            We may now apply \cite[Lemma 4.6]{grillakis1987stability1} to conclude that $A \circ u^s \in C^1([0,\infty),\R)$, and that
            \begin{align*}
                (A \circ u^s)'(t) &= -\jb{J(u^s(t))^* \nabla E(u^s(t)), DA(u^s(t))}\\
                &= - \jb{D E(u^s(t)),J(u^s(t))DA(u^s(t))} = \mathcal{S}(u^s(t)),
            \end{align*}
            whence
            \[
                \abs{A(u^s(t)) - A(\psi(s))} \geq t(E(U_c) - E(\psi(s)))
            \]
            for all $0 < \abs{s} \ll 1$ and $t \in [0,\infty)$ by \eqref{S circ u^s lower bound}. This shows that $A \circ u^s$ is unbounded, but we also have
            \[
                \abs{A(u)} \leq \norm{Z}_\Xspace \norm{M(u) - U_c}_\Xspace \leq \norm{Z}_\Xspace \norm{\iota_{\Wspace \hookrightarrow \Xspace}}\nu_0
            \]
            for every $u \in \tube_{\nu_0}^\Wspace$ by the definition of $A$ in \eqref{def A(u)}. We have arrived at a contradiction, and must conclude that the $U_c$-orbit is unstable.
        \end{proof}

\section{Hamiltonian structure for water waves with a point vortex}
    \label{point vortex formulation section}
    
    With our general machinery in place, we turn to the question of stability of solitary capillary-gravity waves with a submerged point vortex.  The next subsection recalls how this system was formulated by Shatah, Walsh, and Zeng in \cite{shatah2013travelling}.   In Section~\ref{point vortex hamiltonian section}, we show that the problem can be rewritten once more as an abstract Hamiltonian system of the general form \eqref{abstract Hamiltonian system}, and verify that the corresponding energy, momentum, Poisson map, symmetry group, and bound states meet the many requirements of Section~\ref{abstract formulation assumptions section}.
    
    \subsection{Nonlocal formulation}
        \label{nonlocal formulation section}
        Consider the capillary-gravity water wave problem with a point vortex described in \eqref{intro point vortex problem}.   In any simply connected subset of $\Omega_t \setminus \{\bar x\}$, the velocity ${v}$ can be decomposed as 
        \begin{equation}
            \label{splitting of v}
            v = \nabla \Phi + \epsilon\nabla \Theta,
        \end{equation}
        where $\Phi$ is harmonic on $\Omega_t$ and $\Theta$ is harmonic on the subset. The latter represents the vortical contribution of the point vortex. Since the surface is a graph, we will use $\Theta = \Theta_1 - \Theta_2$, where 
        \begin{align*}
            \Theta_1(x)&\coloneqq-\frac{1}{\pi} \arctan{\parn*{\frac{x_1-\bar{x}_1}{\abs{x-\bar{x}}+x_2-\bar{x}_2}}},\\
            \Theta_2(x)&\coloneqq\frac{1}{\pi} \arctan \parn*{\frac{x_1-\bar{x}_1}{\abs{x-\bar{x}'}-x_2-\bar{x}_2}}.
        \end{align*}
        Then $\Theta$ is harmonic on the open set $\brac{x \in \R^2 : x_1 \neq \bar{x}_1 \text{ or } \abs{x_2} < -\bar{x}_2}$, and $\nabla \Theta$ extends to a smooth velocity field on $\Omega_t \setminus \brac{\bar{x}}$. The purpose of $\Theta_2$, which corresponds to a mirror vortex at $\bar{x}' \coloneqq (\bar{x}_1,-\bar{x}_2)$, is to make $\nabla \Theta$ decay faster as $\abs{x} \to \infty$.  Indeed, $\nabla \Theta$ is $L^2$ on the complement of any neighborhood of $\bar{x}$ in $\Omega_t$.
         
        For later use, we also introduce notation for the harmonic conjugate of $\Theta$, which takes the form $\fund = \fund_1 - \fund_2$ with
        \begin{equation}
            \fund_1(x) \coloneqq \frac{1}{2\pi} \log{\abs{x-\bar x}}, \quad \fund_2(x)\coloneqq\frac{1}{2\pi} \log{\abs{x-\bar x'}}.
        \end{equation}
        Note that the convention used here is that $\nabla \Theta = \nabla^\perp \fund$, where we recall that $\nabla^\perp$ is the skew gradient introduced in \eqref{definition of vorticity}.
        
        The rationale behind splitting $v$ according to \eqref{splitting of v} is that it nearly decouples the tasks of determining the rotational and irrotational parts of the velocity.  Indeed, $\Theta$ is entirely explicit given $\bar{x}$, which solves the differential equation \eqref{euler point vortex motion}.  The main analytical challenge is determining $\Phi$ and $\eta$.  But for this we can proceed as in the classical Zakharov--Craig--Sulem formulation of the \emph{irrotational} water wave problem: Because $\Phi$ is harmonic, it is enough to know $\eta$ and the trace
        \[
            \varphi = \varphi(x_1) \coloneqq \Phi(x_1, \eta(x_1))
        \]
        of $\Phi$ on the surface. Notice that $\eta$ and $\varphi$ then have the fixed spatial domain $\R$. The problem can therefore be reduced to the boundary, with the rotational part $\nabla \Theta|_{S_t}$ being viewed as a forcing term.  
        
        On $S_t$, we must ensure that the kinematic condition and Bernoulli condition are satisfied.  Naturally, these will now involve tangential and normal derivatives of $\Phi$ and $\Theta$. Here and in the sequel, we will therefore make use of the shorthands
        \[
            \nonnorm \coloneqq \parn*{ -\eta^\prime \partial_{x_1}  + \partial_{x_2}  }|_{S_t}, \qquad \nontan  \coloneqq \parn*{ \partial_{x_1}  + \eta^\prime \partial_{x_2} }|_{S_t},
        \]
        which arise when parameterizing the free surface using $\eta$.   Note also that we are using the convention that spatial derivatives of quantities restricted to the boundary are denoted by a prime, while $\partial_{x_1}$ is reserved for functions of two or more spatial variables. Exceptions will be made for certain differential operators, when this does not cause ambiguities.
        
        The tangential derivative $\nontan\Phi$ is simply $\varphi'$, but to express the normal derivative $\nonnorm \Phi$ requires using the nonlocal Dirichlet--Neumann operator $\DN(\eta) \colon \dot{H}^{\regindex}(\R) \to \dot{H}^{\regindex-1}(\R)$, which is defined by
        \begin{equation}
            \label{def D-N}
            \DN(\eta)\phi\coloneqq \nonnorm (\mathcal{H}(\eta)\phi),
        \end{equation}
        where $\mathcal{H}(\eta)\phi \in \dot{H}^1(\Omega_t)$ is uniquely determined as the harmonic extension of $\phi \in \dot{H}^{\regindex}(\R) $ to $\Omega_t$.
        
        It is well known that, for any $\regindex_0 > 1$, $\regindex \in [1/2-\regindex_0,1/2+\regindex_0]$, and $\eta \in H^{\regindex_0+1/2}(\R)$, the operator $\DN(\eta)$ is an isomorphism. Moreover, the mapping $\eta \mapsto \DN(\eta)$ is analytic, $\DN(\eta) \colon \dot{H}^{1/2}(\R) \to \dot{H}^{-1/2}(\R)$ is self-adjoint, and $\DN(0) = \abs{\partial_{x_1}}$, the Calder\'on operator.  We refer the reader to \cite[Appendix A]{shatah2013travelling}, \cite[Section 6]{shatah2008geometry}, or \cite[Chapter 3 and Appendix A]{lannes2013book} for more details.  
        
        Finally, using the Dirichlet--Neumann operator, we can rewrite the water wave problem with a point vortex in terms of $(\eta, \varphi, \bar{x})$ as
        \begin{equation}
            \label{SWZ formulation}
            \left\{
            \begin{aligned}
                \partial_t \eta &= \DN(\eta)\varphi + \epsilon\nonnorm\Theta,\\
                \partial_t \varphi &=\begin{multlined}[t]-\frac{(\varphi')^2 -2\eta'\varphi'\DN(\eta)\varphi - (\DN(\eta)\varphi)^2}{2\jb{\eta'}^2}-g\eta + b\parn*{\frac{\eta'}{\jb{\eta'}}}' \\
                 -\epsilon\varphi'\Theta_{x_1}|_S - \frac{\epsilon^2}{2}\parn*{\abs{\nabla \Theta}^2}|_S+ \epsilon \xi|_S\cdot \partial_t \bar{x},
                \end{multlined}\\
                \partial_t \bar{x} &= \nabla \Phi(\bar{x}) -\epsilon \partial_{x_1}\Theta_2(\bar{x})e_1,
            \end{aligned}
            \right.
        \end{equation}
        where, to simplify the notation, we have introduced $\Xi \coloneqq \Theta_1 + \Theta_2$ and $\xi \coloneqq (\Theta_{x_1},\Xi_{x_2})$. The motivation for this choice being that $\nabla_{\bar{x}} \Theta = -\xi$.
        
        The first equation in \eqref{SWZ formulation} is simply the kinematic condition in \eqref{euler boundary condition}, while the second follows from evaluating Bernoulli's law along the free surface using the dynamic condition to replace the trace of the pressure with the (signed) curvature.  Finally, the third equation is just \eqref{euler point vortex motion} in view of the splitting \eqref{splitting of v}. Observe that $\partial_t \bar{x}$ can easily be eliminated from the equation for $\partial_t \varphi$, but we opt not to do so.
    
    \subsection{Hamiltonian formulation}
        \label{point vortex hamiltonian section}
        
        We now endeavor to rewrite \eqref{SWZ formulation} as a Hamiltonian system for the state variable $u = (\eta, \varphi, \bar{x})$.   The first step is to fix a functional analytic framework.  For that, we introduce the continuous scale of spaces
        \begin{equation}
            \label{def X sigma}
            \Xspace^\regindex = \Xspace_1^\regindex \times \Xspace_2^\regindex \times \Xspace_3 \coloneqq {H}^{\regindex+1/2}(\R) \times \parn*{\dot{H}^{\regindex}(\R) \cap \dot{H}^{1/2}(\R) } \times \R^2,  \qquad k \geq 1/2.
        \end{equation}
        For each $\regindex \geq 1/2$, $\Xspace^\regindex$ is a Hilbert space, and the embedding $\Xspace^k \hookrightarrow \Xspace^{k'}$ is dense for all $1/2 \leq k' \leq k$.
        
        For the energy space, we take 
        \begin{equation}
            \label{def point vortex X}
            \Xspace \coloneqq \Xspace^{1/2} = H^1(\R) \times \dot{H}^{1/2}(\R) \times \R^2,
        \end{equation}
        which has the space
        \[
            \Xspace^* = H^{-1}(\R) \times \dot{H}^{-1/2}(\R) \times \R^2
        \]
        as its dual, and for which the isomorphism  $I \colon \Xspace \to \Xspace^*$ takes the explicit form 
        \[
            I = \parn[\big]{1-\partial_{x_1}^2, \abs{\partial_{x_1}}, \Id_{\R^2}}.
        \]
        This choice for $\Xspace$ ensures that $\nabla \Phi \in L^2(\Omega_t)$, and therefore that the kinetic energy corresponding to the irrotational part of the velocity is finite.
        
        On the other hand, anticipating the Dirichlet--Neumann operator, we expect to need $k > 1$ in \eqref{def X sigma} to ensure that the energy is smooth.  With that in mind, set
        \begin{equation}
            \label{def point vortex V}
            \Vspace \coloneqq \Xspace^{\regV+} =  H^{3/2+}(\R) \times \parn*{\dot{H}^{1+}(\R) \cap \dot{H}^{1/2}(\R)}\times \R^2,
        \end{equation}
        where by $\Xspace^{\regV+}$ we mean $\Xspace^{\regV + s}$ for a fixed $0 < s \ll 1$. For the well-posedness space we use
        \begin{equation}
            \label{def point vortex W}
            \Wspace \coloneqq \Xspace^{\regW+} = H^{3+}(\R) \times \parn*{\dot H^{\regW+}(\R) \cap \dot H^{\frac{1}{2}}(\R)} \times \R^2.
        \end{equation}
        A local well-posedness result at this level of regularity was obtained for irrotational capillary-gravity water waves by Alazard, Burq, and Zuily \cite{alazard2011surface}.  While the Cauchy problem for \eqref{SWZ formulation} has not yet been studied, it is reasonable to suppose that local well-posedness will hold in the same space. In our setting, this is the minimal regularity required to have the traces of the velocity be Lipschitz on the surface.
        
        Note that our results hold with any smoother choice of $\Wspace$ as well.  We also mention that very recently Su \cite{su2018long} has obtained long time well-posedness results for gravity waves with a point vortex (corresponding to $b = 0$).  The Gagliardo--Nirenberg interpolation inequality yields the following for our choice of spaces.
        
        \begin{lemma}[Function spaces]
            \label{VWX meet assumptions lemma}
            Let $\Xspace$, $\Vspace$, and $\Wspace$ be defined by \eqref{def point vortex X}, \eqref{def point vortex V}, and \eqref{def point vortex W}, respectively. Then there exists a constant $C > 0$ and $\theta \in (0,1/4)$ such that Assumption~\ref{abstract interpolation assumption} is satisfied.
        \end{lemma}
        
        Lastly, recall that for the problem to be well-defined, the surface must lie between the point vortex at $\bar{x} \in \Omega_t$, and its mirror at $\bar{x}'$.  We therefore let
        \[
            \nbhdO \coloneqq \brac{u \in \Xspace : \bar{x}_2 < \eta(\bar{x}_1) < -\bar{x}_2},
        \]
        and seek solutions taking values in $\nbhdO \cap \Wspace$ at each time.
        
        We endow $\Xspace$ with symplectic structure by prescribing a Poisson map.  First, consider the linear operator $\hat J \colon \Dom(\hat J) \subset \Xspace^* \to \Xspace$ defined by 
        \begin{equation}
            \label{def hat J}
            \hat J \coloneqq
            \begin{pmatrix}
                0 & 1 & 0 & 0 \\
                -1 & 0 & 0 & 0 \\
                0 & 0 & 0 & \epsilon^{-1} \\
                0 & 0 & -\epsilon^{-1} & 0
            \end{pmatrix},
        \end{equation}
        with the natural domain
        \[
            \Dom(\hat J) \coloneqq \parn*{H^{-1}(\R) \cap \dot{H}^{1/2}(\R)} \times \parn*{H^1(\R) \cap \dot{H}^{-1/2}(\R)} \times \R^2.
        \]
        One can understand $\hat J$ as encoding the Hamiltonian structure for the point vortex and water wave in isolation.  To get the full system, we must incorporate wave-vortex interaction terms. For each $u \in \nbhdO \cap \Vspace$,  define
        \begin{equation}
            \label{def B(u)}
            \Lin(\Xspace) \ni B(u) \coloneqq \Id_{\Xspace} + \mathcal{K}(u),
        \end{equation}
        where $\mathcal{K}(u) \in \Lin(\Xspace)$ is the finite-rank operator given by
            \[
            \mathcal{K}(u) \dot w \coloneqq
            \begin{pmatrix}
                0 & 0 & 0 & 0\\ 
                -\epsilon \Xi_{x_2}|_S & \epsilon \Theta_{x_1}|_S & \epsilon \Theta_{x_1}|_S & \epsilon \Xi_{x_2}|_S \\
                0 & 1 & 0 & 0 \\
                -1 & 0 & 0 & 0
            \end{pmatrix} 
            \begin{bmatrix}
                \jb{\Theta_{x_1}|_S,\dot{\eta}} \\
                \jb{\Xi_{x_2}|_S,\dot{\eta}}  \\
                \dot{\bar{x}}_1 \\
                \dot{\bar{x}}_2
            \end{bmatrix},
        \]
        for all $\dot w \in \Xspace$.  The full Poisson map is formed, like in \eqref{abstract state dependent poisson map}, by composing $\hat J$ with $B(u)$:
        
        \begin{lemma}[Properties of $J$]
            \label{J meets assumptions lemma}
            For each $u \in \nbhdO \cap \Vspace$, the operator $J(u)\colon \Dom(\hat J) \subset \Xspace^* \to \Xspace$ is given by
            \begin{equation}
                \label{def J}
                J(u) \coloneqq B(u)\hat{J} =
                \begin{pmatrix} 
                    0 & 1 & 0 & 0 \\
                    -1 & {J}_{22} & {J}_{23} &  {J}_{24} \\
                    0 & {J}_{32} & 0 & \epsilon^{-1} \\
                    0 & {J}_{42} & -\epsilon^{-1} & 0
                \end{pmatrix},
            \end{equation}
            where the entries are given by
            \begin{align*}
                J_{22} &= -\epsilon \Xi_{x_2}|_S \jb{\placeholder,\Theta_{x_1}} + \epsilon \Theta_{x_1} \jb{\placeholder,\Xi_{x_2}|_S},\\
                J_{23} &= -\Xi_{x_2}|_S, \\
                J_{24} &=  \Theta_{x_1}|_S, \\
                J_{32} &= \jb{\placeholder,\Xi_{x_2}|_S},\\
                J_{42} &= -\jb{\placeholder,\Theta_{x_1}|_S},
            \end{align*}
            and Assumption~\ref{abstract symplectic assumption} is satisfied.
        \end{lemma}
        \begin{proof}
            It is clear from its definition in \eqref{def hat J} that $\hat J$ is injective and closed, and its domain $\Dom(\hat J)$ is dense in $\Xspace^*$ by Lemma~\ref{sobolev density lemma}.  Thus parts \ref{J densely defined assumption} and \ref{J injectivity assumption} of Assumption~\ref{abstract symplectic assumption} hold.  Now, fix $u \in \nbhdO \cap \Vspace$ and consider the operator $B(u)$ given by \eqref{def B(u)}. The map $\mathcal{K}(u)$ has finite rank, so $B(u)$ is a compact perturbation of identity. In particular, $B(u)$ is Fredholm of index $0$. On the other hand, $B(u)$ is clearly injective, and thus it must be an isomorphism on $\Xspace$.  This proves part~\ref{B bijective assumption}.  The properties of the mapping $u \mapsto B(u)$ asked for in part~\ref{regularity of J} are obvious from the definition \eqref{def B(u)}.  Finally, the skew-adjointness of $J(u)$ is apparent from the formula \eqref{def J}.
        \end{proof}
        
        Next, we must determine the energy associated to a water wave with a point vortex.  Classically, the kinetic energy is given by  $\tfrac{1}{2}\int \abs{v(t)}^2 \, dx$.  To adapt this to the point vortex case, we use the splitting \eqref{splitting of v} and formally integrate by parts. This produces traces on $S_t$, plus terms at the vortex center. We neglect the singular one, corresponding to $\Gamma_1$, which is equivalent to removing the self-advection of the point vortex as in the Helmholtz--Kirchhoff model. Ultimately, this leads us to define the energy functional $\eng = \eng(u)$ to be 
        \begin{equation}
            \label{defE}
            \eng(u) \coloneqq \kinE(u)  + \potE(u),
        \end{equation}
        where
        \begin{equation}
            \label{defK}
            \begin{aligned}
                \kinE(u) &\coloneqq K_0(u) + \epsilon K_1(u) + \epsilon^2 K_2(u)\\
                &\coloneqq \frac{1}{2} \int_\R \varphi \DN(\eta) \varphi \, dx_1
                + \epsilon \int_\R \varphi \nonnorm \Theta \, dx_1
                +  \frac{1}{2} \epsilon^2 \parn*{\int_{\R} \sTheta \nonnorm \Theta \, dx_1 + \fund_2(\bar{x})} \\
            \end{aligned}
        \end{equation}
        is the kinetic energy, and
        \begin{equation}
            \label{defV}
            \potE(u) \coloneqq \int_\R \parn*{\frac{1}{2}g\eta^2 + b(\jb{\eta'}-1)}\,dx_1
        \end{equation}
        is the potential energy. Notice that $\potE$ depends solely on the surface profile. A similar procedure also shows that
        \begin{equation}
            \label{defP}
            \mom = \mom(u) \coloneqq \epsilon \bar{x}_2 - \int_\R \eta' \parn*{\varphi + \epsilon \Theta|_S } \, dx_1.
        \end{equation}
        is the momentum carried by a water wave with a submerged point vortex.
        
        It is easy to see that $\eng,\mom \in C^\infty(\nbhdO \cap \Vspace; \R)$.  For the convenience of the reader, the first and second Fr\'echet derivatives of $\eng$ and $\mom$ are recorded in Appendix~\ref{variations appendix}.  By inspection, we see that $D\eng$ and $D \mom$ admit the explicit extensions
        \begin{align}
            \label{formula nabla E}
            \nabla \eng(u) &\coloneqq (\eng_\eta'(u),\eng_\varphi'(u),\nabla_{\bar{x}}\eng(u)),\\
            \label{formula nabla P}
            \nabla \mom(u) &\coloneqq (\mom_\eta'(u),\mom_\varphi'(u),\nabla_{\bar{x}}\mom(u)),
        \end{align}
        in $C^\infty(\nbhdO \cap \Vspace; \Xspace^*)$, with 
        
        \begin{equation}
            \label{gradient of E}
            \begin{aligned}
                \eng_\eta'(u) &\coloneqq\begin{multlined}[t]\frac{(\varphi')^2 -2\eta'\varphi'\DN(\eta)\varphi - (\DN(\eta)\varphi)^2}{2\jb{\eta'}^2}+g\eta - b\parn*{\frac{\eta'}{\jb{\eta'}}}' \\
                +\epsilon\varphi'\Theta_{x_1}|_{S_t} + \frac{\epsilon^2}{2}\parn*{\abs{\nabla \Theta}^2}|_{S_t},
                \end{multlined}\\
                \eng_\varphi'(u) &\coloneqq \DN(\eta)\varphi +\epsilon \nonnorm \Theta,\\
                \nabla_{\bar{x}} \eng(u) &\coloneqq - \frac{1}{2}\epsilon^2 \int_{\R} \nonnorm(\Theta \xi)\,dx - \epsilon \int_{\R} \varphi \nonnorm \xi \,dx_1 -\epsilon^2 \partial_{x_1}\Theta_2(\bar{x})e_2,
            \end{aligned}
        \end{equation}
        and 
        \begin{equation}
            \label{gradient of P}
            \begin{aligned}
            \mom_\eta'(u) &\coloneqq \varphi' + \epsilon \Theta_{x_1}|_{S_t},\\
            \mom_\varphi'(u) &\coloneqq -\eta',\\
            \nabla_{\bar{x}} \mom(u)&\coloneqq \epsilon e_2 + \epsilon\int_{\R} \eta' \xi|_{S_t} \,dx_1.
            \end{aligned}
        \end{equation}
        Thus Assumption~\ref{extend DP and DE assumption} is indeed satisfied.
        
        The next lemma confirms that the Hamiltonian system for this choice of the energy and Poisson map corresponds to the water wave with a point vortex problem in \eqref{SWZ formulation}.  
        
        \begin{theorem}[Hamiltonian formulation]
            \label{hamiltonian theorem}
            A function $u \coloneqq (\eta, \varphi, \bar{x}) \in C^1([0,t_0); \Wspace \cap \nbhdO)$ is a solution of the capillary-gravity water wave problem with a point vortex \eqref{SWZ formulation} if and only if it is a solution to the abstract Hamiltonian system
            \begin{equation}
                \label{hamiltonian formulation}
                \frac{du}{dt}  = J(u) D\eng (u),
            \end{equation}
            where $J = J(u)$ is the Poisson map \eqref{def J} and $\eng$ is the energy functional defined in~\eqref{defE}.
        \end{theorem}
        \begin{proof}
            Written out more explicitly using \eqref{def J}, the Hamiltonian system \eqref{hamiltonian formulation} is
            \begin{equation}
                \label{point vortex problem}
                \left\{
                    \begin{aligned}
                        \partial_t \eta &= \eng_\varphi'(u)\\
                        \partial_t \varphi &= \begin{multlined}[t]-\eng_\eta'(u)\\
                        + \epsilon \xi|_{S_t} \cdot \parn*{\jb{\eng_\varphi'(u),\Xi_{x_2}}+\epsilon^{-1}\partial_{\bar x_2}\eng(u),-\jb{\eng_\varphi'(u),\Theta_{x_1}}-\epsilon^{-1}\partial_{\bar x_1}\eng(u)}
                        \end{multlined}\\
                        \partial_t \bar{x} &= \parn*{\jb{\eng_\varphi'(u),\Xi_{x_2}}+\epsilon^{-1}\partial_{\bar x_2}\eng(u),-\jb{\eng_\varphi'(u),\Theta_{x_1}}-\epsilon^{-1}\partial_{\bar x_1}\eng(u)}
                    \end{aligned}
                \right.
            \end{equation}
            Using \eqref{gradient of E}, we see that the first of Hamilton's equations is 
            \[
                \partial_t \eta=\DN(\eta)\varphi+\epsilon\nonnorm \Theta,
            \]
            which is equivalent to the kinematic boundary condition in \eqref{SWZ formulation}. Moreover, the equation for $\partial_t \varphi$ above agrees with the corresponding one in \eqref{SWZ formulation}, as the final term is simply $\epsilon \xi|_{S_t}\cdot\partial_t \bar{x}$ by the third equation in \eqref{point vortex problem}.
            
            Only the equation for the motion of the point vortex remains. Written out explicitly, we find that
            \begin{align*}
                \partial_t \bar{x}_2 &= \int_{\R} (\varphi \nonnorm \Theta_{x_1} - \Theta_{x_1}|_{S_t} \DN(\eta)\varphi)\,dx_1 + \frac{\epsilon}{2}\int_{\R} (\Theta|_{S_t} \nonnorm \Theta_{x_1} - \Theta_{x_1}|_{S_t} \nonnorm \Theta)\,dx_1\\
                &= \int_{S_t} N \cdot (\Gamma_{x_1} \nabla \Psi - \Psi\nabla\Gamma_{x_1})\,dS + \frac{\epsilon}{2}\int_{S_t} N \cdot (\Gamma_{x_1} \nabla \Gamma - \Gamma \nabla \Gamma_{x_1})\,dS,
            \end{align*}
            where $\Psi$ is the harmonic conjugate to $\Phi$ in $\Omega_t$ and $N$ is the outward-pointing unit normal. Now, owing to the fact that $\Gamma$ and $\Gamma_{x_1}$ are harmonic on $\R^2 \setminus \{\bar{x},\bar{x}'\}$, the final integral is equal to
            \[
                \int_{{x_2 = 0}} (\Gamma_{x_1}\Gamma_{x_2}- \Gamma \Gamma_{x_1 x_2})\,dx_1 = 0
            \]
            by path independence. Here we have used that $\Gamma = \Gamma_{x_1} = 0$ on $\{x_2 = 0\}$.
            
            On the other hand, we have the identity
            \[
                \int_{S_t} N \cdot (\Gamma_{x_1} \nabla \Psi - \Psi\nabla\Gamma_{x_1})\,dS =  \int_{\abs{x-\bar{x}}=r} N \cdot (\Gamma_{x_1} \nabla \Psi - \Psi\nabla\Gamma_{x_1})\,dS
            \]
            for all $0 < r \ll 1$. Notice that $\Gamma_2$ is harmonic in $\Omega_t$, so only $\Gamma_1$ contributes in the limit $r \to 0$. Setting $x -\bar{x}=re^{i\theta}$, under the natural identification, we have
            \begin{align*}
                \int_{\abs{x-\bar{x}}=r} \Gamma_{1,x_1} N \cdot \nabla \Psi\,dS &= \int_0^{2\pi} \frac{\cos(\theta)}{2\pi r}(\cos(\theta),\sin(\theta)) \cdot \nabla \Psi(\bar{x}+re^{i\theta})\,d\theta\\
                &= \frac{1}{4\pi} \int_0^{2\pi} (1+\cos(\theta/2),\sin(\theta/2))\cdot \nabla \Psi(\bar{x}+re^{i\theta})\,d\theta
            \end{align*}
            and
            \begin{align*}
                -\int_{\abs{x-\bar{x}}=r}\Psi N\cdot\nabla\Gamma_{1,x_1}\,dS &= \frac{1}{2\pi} \int_0^{2\pi} \cos(\theta)\frac{\Psi(\bar{x}+re^{i\theta})}{r}\,d\theta\\
                &=\frac{1}{4\pi} \int\limits_0^{2\pi} \int\limits_0^1 (1+\cos(\theta/2),\sin(\theta/2))\cdot \nabla \Psi(\bar{x}+tre^{i\theta})\,dt\,d\theta.
            \end{align*}
            
            In total, then,  
            \[
                \partial_t \bar x_2  = \lim_{r \to 0} \int_{\abs{x-\bar{x}}=r} N \cdot (\Gamma_{x_1} \nabla \Psi - \Psi\nabla\Gamma_{x_1})\,dS = \Psi_{x_1}(\bar{x}),
            \]
            and an essentially identical argument shows that
            \[
                \partial_t \bar{x}_1=-\partial_{x_2}\Psi(\bar x)-\epsilon \partial_{x_1}\Theta_2(\bar{x}).
            \]
            Recalling that $\Psi$ and $\Phi$ are harmonic conjugates, these two equations are equivalent to the vortex dynamics equation in \eqref{SWZ formulation}.
            
            Finally, conservation of the energy $\eng$ is immediate from the fact that $u$ is a $C^1$ solution of \eqref{hamiltonian formulation}. The conservation of momentum $P$ is simply a consequence of  \eqref{abstract T'(0) and P' identity}, which we verify below in Lemma~\ref{properties of T lemma}.  
        \end{proof}
    
    \subsection{Symmetry}
        
        Let $T = T(s)\colon \Xspace \to \Xspace$ be the one-parameter family of affine mappings given by 
        \begin{equation}
            \label{def symmetry T}  
            T(s)u \coloneqq (\eta(\placeholder-s),\varphi(\placeholder-s),\bar{x}+se_1),\qquad s \in \R,
        \end{equation}
        representing the invariance of the underlying system with respect to horizontal translations. The linear part of the family is
        \begin{equation}
            \label{def dT(s)}
            dT(s)u = (\eta(\placeholder-s),\varphi(\placeholder-s),\bar{x}), \qquad s \in \R,
        \end{equation}
        and the infinitesimal generator of $T$ is the affine operator
        \begin{equation}
            \label{def T'(0)}
            T'(0) = dT'(0) + T'(0)0 = -(\partial_{x_1}, \partial_{x_1},0) + (0,0,e_1),
        \end{equation}
        with domain $\Dom(T'(0)) = \Xspace^{3/2}$.
        
        \begin{lemma}[Properties of $T$]
            \label{properties of T lemma}
            The group $T(\placeholder)$ satisfies Assumption~\ref{abstract symmetry assumption}.   
        \end{lemma}
        \begin{proof}
            Parts \ref{invariances}, \ref{group flow property}, and \ref{unitary assumption} are obvious from the definition  of $T$. The strong continuity of the group in the respective spaces is likewise straightforward.  Observe also that $T(t)0 = t(0,0,e_1)$, which has norm $\abs{t}$ in both $\Xspace$ and $\Wspace$. Thus part~\ref{affine bound assumption} holds with $\omega(t) = t$.
            
            For part \ref{commutativity assumption}, note that $dT(s)$ is invariant on $I^{-1}\Dom(\hat{J})$, which is therefore the common domain of definition for both sides at the top of \eqref{abstract commutation identity}. Verifying that we have equality in the two equations for all $s \in \R$ is then just a matter of inserting the definitions. For part \ref{T'(0) assumption}, observe that $\Dom(T^\prime(0)|_{\Vspace}) = \Xspace^{2+}$.  That $\nabla P(u) \in \Dom(\hat J)$ for any $u \in \nbhdO \cap \Dom(T^\prime(0)|_{\Vspace})$ follows from its formula in \eqref{formula nabla P} and \eqref{gradient of P}.  Moreover, \eqref{abstract T'(0) and P' identity} and \eqref{abstract derivative commutation identity} can be obtained by direct computation.
            
            To verify part~\ref{range density}, note that
            \[
                \Rng{\hat{J}} =  (H^1(\R) \cap \dot{H}^{-1/2}) \times (H^{-1}(\R) \cap \dot{H}^{1/2}) \times \R^2, \qquad \Dom(T^\prime(0)|_{\Wspace}) = \Xspace^{7/2}, 
            \]
            so
            \[
                \Dom(T^\prime(0)|_{\Wspace}) \cap \Rng{\hat{J}} = (H^4(\R) \cap \dot{H}^{-1/2}) \times (H^{-1}(\R) \cap \dot{H}^{7/2}(\R) \cap \dot{H}^{1/2}) \times \R^2,
            \]
            which is certainly dense in $\Xspace$ (cf. Lemma~\ref{sobolev density lemma}). Finally, the conservation of energy under the group \ref{T conserves energy} is immediate given the translation invariant nature of $\eng$ in \eqref{defE}, \eqref{defK}, and \eqref{defV}.
        \end{proof}
        
    \subsection{Traveling waves}
        
        In Theorem \ref{existence theorem}, we prove the existence of a surface of small-amplitude traveling wave solutions of the point vortex problem, parameterized by the vortex strength $\epsilon$ and the depth of the point vortex $a$. For the stability analysis, however, it is important to fix $\epsilon$, as it appears as part of the equation.  We will therefore consider the families
        \begin{equation}
            \label{point vortex U_c}
            \mathscr{C}_{\mathcal{I}}^\epsilon \coloneqq \{U_{c(\epsilon,a)} \coloneqq (\eta(\epsilon,a), \varphi(\epsilon,a), -ae_2) : a \in \mathcal{I}\} \subset \nbhdO \cap \Wspace
        \end{equation}
        of traveling water waves with a point vortex of strength $\epsilon$ at $-ae_2$, traveling at speed $c(\epsilon,a)$, for nontrivial compact intervals $\mathcal{I} \subset (0,\infty)$ and $0 < \epsilon \ll 1$. From \eqref{asymptotic form waves} we see that $a \mapsto c(\epsilon,a)$ is a diffeomorphism onto its image when $\epsilon \neq 0$ is sufficiently small, which justifies viewing $\mathscr{C}_{\mathcal{I}}^\epsilon$ as being parameterized by the wave speed $c$.
        
        \begin{lemma}
            \label{Uc bound state lemma}
            For each nontrivial compact intervals $\mathcal{I} \subset (0,\infty)$ and $0 < \epsilon \ll 1$, the family $\mathscr{C}_\mathcal{I}^\epsilon$ satisfies Assumption~\ref{bound state assumption}.
        \end{lemma}
        \begin{proof}
            From the construction of $\mathscr{C}_{\mathcal{I}}^\epsilon$ in Theorem~\ref{existence theorem}, we know that the mapping $c \mapsto U_c$ is of class $C^1$. Since the existence theory can be carried out for any $\regindex > 3/2$, we can ensure that $U_c$ and $\frac{dU_c}{dc}$ satisfy Assumption~\ref{bound state assumption}\ref{bound state domain assumption}. Also, the non-degeneracy condition \ref{bound state non-degeneracy} holds for small enough $\epsilon$ in view of \eqref{asymptotic form waves}. Finally,
            \[
                \norm{T(s)U_c - U_c}_{\Xspace} \geq \abs{(se_1 - ae_2) - ae_2} = \abs{s},
            \]
            so the second option in \ref{non-periodic bound state assumption} holds.
        \end{proof}
        
        Formally, the traveling waves on $\mathscr{C}_\mathcal{I}^\epsilon$ are stable if we can show that the moment of instability defined in \eqref{abstract d definition} has positive second derivative. This can be shown to be the case when $\epsilon$ is small.
        \begin{lemma}
            \label{formal stability lemma}
            Fix a nontrivial compact interval $\mathcal{I} \subset (0,\infty)$. Then
            \[
                d''(c(\epsilon,a)) > 0, \qquad \text{for all $a \in \mathcal{I}$},
            \]
            when $0 < \abs{\epsilon} \ll 1$.
        \end{lemma}
        \begin{proof}
            By \eqref{abstract d'' identity},
            \begin{equation}
                \label{point vortex d'' identity}
                \partial_ac(\epsilon,a)d''(c(\epsilon,a)) = - \jb{D\mom(U_{c(\epsilon,a)}),\partial_a U_{c(\epsilon,a)}},
            \end{equation}
            and from Appendix~\ref{existence theory appendix}
            \begin{align*}
                c(\epsilon,a_0) &= - \frac{1}{4\pi a}\epsilon + O(\epsilon^3)\\
                U_c(\epsilon,a_0) &= (0,0,-ae_2) + (\eta_2(a),0,0)\epsilon^2 + O(\epsilon^3)
            \end{align*}
            in $C^1(\mathcal{I};\Wspace)$. From the latter expression, and \eqref{gradient of P}, we find that
            \[
                \nabla \mom(U_{c(\epsilon,a)}) = (\Theta_{x_1}(\cdot,0),0,e_2)\epsilon - (0,\eta_2'(a),0)\epsilon^2 + O(\epsilon^3)
            \]
            in $C^0(\mathcal{I},\Xspace^*)$, and we can finally deduce from \eqref{point vortex d'' identity} that
            \[
                \parn*{\frac{1}{4\pi a^2}\epsilon + O(\epsilon^3)}d''(c(\epsilon,a))=\epsilon + O(\epsilon^3)
            \]
            or
            \[
                d''(c(\epsilon,a)) = 4\pi a^2 +O(\epsilon^2)
            \]
            in $C^0(\mathcal{I},\R)$. The right hand side is positive on $\mathcal{I}$ for sufficiently small $\epsilon \neq 0$.
        \end{proof}

\section{Stability of solitary waves with a point vortex}
    \label{point vortex spectrum section}
    
    In the previous section, we confirmed that the capillary-gravity water wave problem with a point vortex has a Hamiltonian formulation \eqref{hamiltonian formulation} that is invariant under the translation group $T(\placeholder)$ defined in \eqref{def symmetry T}, and we introduced the corresponding trio of Banach spaces $\Wspace \hookrightarrow \Vspace \hookrightarrow\Xspace$ in \eqref{def point vortex X}--\eqref{def point vortex W}.   We are now prepared to state and prove the main theorem:
    
    \begin{theorem}[Main theorem]
        \label{main point vortex theorem}
        Fix a nontrivial compact interval $\mathcal{I} \subset (0,\infty)$ and $0 < \abs{\epsilon} \ll 1$. Then the family $\mathscr{C}_\mathcal{I}^\epsilon$ of solitary capillary-gravity water waves with a submerged point vortex are conditionally orbitally stable in the sense of Theorem~\ref{abstract stability theorem}.
    \end{theorem}
    
    It is important at this point to emphasize that the family $\mathscr{C}_{\mathcal{I}}^\epsilon$ comprises \emph{all} even traveling wave solutions of \eqref{SWZ formulation} with $(\eta, \varphi, c)$ and $\epsilon$ in a neighborhood of $0$, in a certain function space setting; see Appendix~\ref{existence theory appendix} for more details. Thus, the stability furnished by Theorem~\ref{main point vortex theorem} applies to \emph{all} even waves that are sufficiently small-amplitude, slow moving, and have small enough vortex strength.
    
    We have already addressed a number of the hypotheses of the general theory.  Moreover, Lemma~\ref{formal stability lemma} shows that the family $\mathscr{C}_{\mathcal{I}}^\epsilon$ is \emph{formally} orbitally stable for $0 < \abs{\epsilon} \ll 1$.  The only remaining task --- which is by far the most difficult --- is to verify that the waves in $\mathscr{C}_{\mathcal{I}}^\epsilon$ lie at a saddle point of the energy with a one-dimensional negative subspace, as required by Assumption~\ref{spectral assumptions}.   Our basic approach follows along the lines of Mielke's study of the irrotational case \cite{mielke2002energetic}, with many modifications necessitated by the presence of the point vortex.  
    
    Recall that the family of traveling waves $\{U_c\}$ are critical points of the augmented Hamiltonian $\augHam \coloneqq \eng - c\mom$. Because $\varphi$ occurs quadratically in $E$, and
    \[
        \jb{D_\varphi\augHam,\dot{\varphi}} = \int_{\R} \dot{\varphi}(\DN(\eta)\varphi+\epsilon\nonnorm\Theta + c\eta')\,dx_1,
    \]
    we can eliminate $\varphi$ by introducing
    \begin{equation}
        \label{varphi*formula}
        \begin{aligned}
            \varphi_*(v) &\coloneqq - \DN(\eta)^{-1}(c\eta' +\epsilon\nonnorm\Theta),\\
            u_*(v) &\coloneqq (\eta,\varphi_*(v),\bar{x}) \in \Vspace,
        \end{aligned}
    \end{equation}
    and the \emph{augmented potential}
    \begin{equation}
        \label{augV def}
        \augV(v) \coloneqq \min_{\varphi \in \Vspace_2}  \augHam(\eta, \varphi, \bar{x}) = \augHam(u_*(v)).
    \end{equation}
    for $v = (\eta, \bar{x}) \in \Vspace_{1,3} \cap \nbhdO_{1,3}$. Here
    \[
        \Vspace_{1,3} \coloneqq \Vspace_1 \times \Vspace_3, \qquad \nbhdO_{1,3} \coloneqq \{(\eta,\bar{x}) \in \Xspace_1 \times \Xspace_3 : \bar{x}_2 < \eta(\bar{x}_1) < - \bar{x}_2\}.
    \]
    Note that $\varphi_* \in C^\infty(\Vspace_{1,3} \cap \nbhdO_{1,3};\Xspace_2^{3/2+})$ and $u_* \in C^\infty(\Vspace_{1,3} \cap \nbhdO_{1,3};\Vspace)$, whence in particular $\augV \in C^\infty(\Vspace_{1,3};\R)$.
    
    For later use, we also define
    \[
        \mathfrak{a} = \mathfrak{a}(v) \coloneqq \parn*{\nabla{(\mathcal{H}\varphi_*)}}|_S, \qquad \mathfrak{b} = \mathfrak{b}(v) \coloneqq \mathfrak{a} + \epsilon \nabla \Theta|_S - ce_1.
    \]
    Thus  $\mathfrak{a}$ is the irrotational part of the velocity field, and $\mathfrak{b}$ is the relative velocity field, both restricted to the surface. Observe that $\mathfrak{b}_2 = \eta' \mathfrak{b}_1$ by \eqref{varphi*formula}.  Because we are working with the steady problem, in what follows we simply write $S$ rather than $S_t$.  
    
    \begin{lemma}
        \label{simplifiedaugVlemma}
        For all $v \in \Vspace_{1,3} \cap \nbhdO_{1,3}$ and $\dot{v} = (\dot{\eta},\dot{\bar{x}}) \in \Vspace_{1,3}$, we have
        \begin{multline}
            \label{simplifedaugVformula}
            \jb{D^2\augV(v)\dot{v},\dot{v}}_{\Vspace_{1,3}^* \times \Vspace_{1,3}} = \jb{D_v^2 \augHam(u_*(v))\dot{v},\dot{v}}_{\Vspace_{1,3}^* \times \Vspace_{1,3}}\\
            - \jb{\mathcal{L}(v)\dot{v},\DN(\eta)^{-1}\mathcal{L}(v)\dot{v}}_{\Xspace_2^* \times \Xspace_2},
        \end{multline}
        where
        \begin{equation}
            \label{def L operator}
            \mathcal{L}(v) \dot{v} \coloneqq \DN(\eta)(\mathfrak{a}_2 \dot{\eta}) + (\mathfrak{b}_1\dot{\eta})' + \epsilon \nonnorm \xi \cdot \dot{\bar{x}}
        \end{equation}
        defines a bounded linear operator $\mathcal{L}(v) \in \Lin(\Xspace_{1,3};\Xspace_2^*)$.
    \end{lemma}
    \begin{proof}
        By the  definitions of $\varphi_*$ in \eqref{varphi*formula} and $\augV$ in \eqref{augV def}, it follows that
        \[
            \jb{D\augV(v),\dot{v}} = \jb{D_\varphi \augHam(u_*(v)),\dot{v}} + \jb{D_v \augHam(u_*(v)),\dot{v}} = \jb{D_v \augHam(u_*(v)),\dot{v}},
        \]
        and
        \begin{align*}
            \jb{D^2\augV(v)\dot{v},\dot{v}} &= \jb{D_vD_\varphi \augHam(u_*(v))\jb{D\varphi_*(v),\dot{v}},\dot{v}} + \jb{D_v^2 \augHam(u_*(v))\dot{v},\dot{v}}\\
            &= - \jb{D_\varphi^2\augHam(u_*(v))\jb{D\varphi_*(v),\dot{v}},\jb{D\varphi_*(v),\dot{v}}}  + \jb{D_v^2 \augHam(u_*(v))\dot{v},\dot{v}},
        \end{align*}
        which yields the claimed formula after computing that
        \begin{align*}
            \DN(\eta) \jb{D\varphi_*(v),\dot{v}}&= -\jb{D_\eta \DN(\eta)\dot{\eta},\varphi_*(v)} + ([\epsilon \Theta_{x_1}|_S - c]\dot{\eta})'+\epsilon \nonnorm\xi \cdot \dot{\bar{x}}\\
            &=\DN(\eta)(\mathfrak{a}_2 \dot{\eta}) + (\mathfrak{b}_1\dot{\eta})' + \epsilon \nonnorm \xi \cdot \dot{\bar{x}}.
            \qedhere
        \end{align*}
    \end{proof}
    
    The next lemma further unpacks the expression \eqref{simplifedaugVformula} to obtain a quadratic form representation on the energy space, in preparation for the verification of Assumption~\ref{spectral assumptions}.
    \begin{lemma}[Extension of $D^2 \augV$]
        \label{characterization of augV lemma}
        For all $v \in \Vspace_{1,3} \cap \nbhdO_{1,3}$, there is a self-adjoint linear operator $A(v) \in \Lin(\Xspace_{1,3};\Xspace_{1,3}^*)$ such that
        \[
            \jb{D^2 \augV(v)\dot{v},\dot{w}}_{\Vspace_{1,3}^* \times \Vspace_{1,3}} = \jb{A(v)\dot{v},\dot{w}}_{\Xspace_{1,3}^* \times \Xspace_{1,3}}
        \]
        for all $\dot{v},\dot{w} \in \Vspace_{1,3}$.  Explicitly,
        \begin{equation}
            \label{defA}
            A = \begin{pmatrix}
                A_{11} & A_{13} \\
                A_{13}^* & A_{33}
            \end{pmatrix},
        \end{equation}
        with entries given by
        \begin{align*}
            A_{11}\dot{\eta} &\coloneqq \parn*{g + \mathfrak{b}_2'\mathfrak{b}_1}\dot{\eta} - \parn*{\frac{b}{\jb{\eta'}^3}\dot{\eta}'}' - \mathcal{M}\dot{\eta},\\
            A_{13} \dot{\bar{x}} &\coloneqq \epsilon \mathfrak{b}_1\nontan(\DN(\eta)^{-1}\nonnorm \xi-\xi) \cdot \dot{\bar{x}} \\
            A_{13}^*\dot{\eta} &\coloneqq \epsilon \int_{\R} \dot{\eta}\mathfrak{b}_1\nontan(\DN(\eta)^{-1}\nonnorm \xi-\xi) \,dx_1,\\
            A_{33} &\coloneqq D_{\bar{x}}^2 \augHam(u_*) + \epsilon^2 \int_{\R} \nonnorm\xi \odot \DN(\eta)^{-1} \nonnorm\xi\,dx_1.
        \end{align*}
        Here $\mathcal{M} \dot\eta\coloneqq - \mathfrak{b}_1(\DN(\eta)^{-1}(\mathfrak{b}_1 \dot\eta)')'$, $x \odot y \coloneqq (x \otimes y + y \otimes x)/2$ is the symmetric outer product, and an explicit expression for $D_{\bar{x}}^2 \augHam(u_*)$ is given in \eqref{explicit second variation bar x}.
    \end{lemma}
    \begin{proof}
        Due to symmetry, it is sufficient to consider the diagonal. A series of rather lengthy, but direct, computations show that
        \begin{multline*}
            \int_{\R} (\mathcal{L}(v)\dot{v})\DN(\eta)^{-1}\mathcal{L}(v)\dot{v}\,dx_1 = \int_{\R} \mathfrak{a}_2 \dot{\eta} \DN(\eta)(\mathfrak{a}_2\dot{\eta})\,dx_1 + \int_{\R} \dot{\eta} \mathcal{M} \dot{\eta}\,dx_1\\
            + \int_{\R} (\mathfrak{a}_2\mathfrak{b}_1'-\mathfrak{a}_2'\mathfrak{b}_1)\dot{\eta}^2\,dx_1 + 2\epsilon \dot{\bar{x}} \cdot \int_{\R} (\mathfrak{a}_2\nonnorm\xi - \mathfrak{b}_1 (\DN(\eta)^{-1}\nonnorm \xi)')\dot{\eta}\,dx_1\\
            + \epsilon^2 \dot{\bar{x}}^T\parn*{\int_{\R} \nonnorm\xi \odot \DN(\eta)^{-1} \nonnorm\xi\,dx_1}\dot{\bar{x}},
        \end{multline*}
        while
        \begin{align*}
            \jb{D_\eta^2 \augHam(u_*)\dot{\eta},\dot{\eta}} &= \begin{multlined}[t]\int_{\R} \mathfrak{a}_2\dot{\eta}\DN(\eta)(\mathfrak{a}_2\dot{\eta})\,dx_1+ \int_{\R} \parn*{g + \epsilon\mathfrak{b}_1 \nontan \Theta_{x_2} + \mathfrak{a}_2\mathfrak{b}_1'}\dot{\eta}^2\,dx_1\\
            + \int_{\R} \frac{b}{\jb{\eta'}^3} (\dot{\eta}')^2\,dx_1,
            \end{multlined}\\
            \nabla_{\bar{x}}\jb{D_\eta\augHam(u_*),\dot{\eta}} &= \epsilon \int_{\R}\parn*{\mathfrak{a}_2\nonnorm\xi - \mathfrak{b}_1 \nontan \xi }\dot{\eta}\,dx_1,
        \end{align*}
        and
        \begin{multline}
            \label{explicit second variation bar x}
            D_{\bar{x}}^2 \augHam(u_*) =2\epsilon^2 D_x^2 \Gamma_2(\bar{x}) - \epsilon \int_{\R} (\DN(\eta)\varphi_* D_{\bar{x}}^2 \Theta + \varphi_*' D_{\bar{x}}^2 \Gamma)|_S\,dx_1 + \epsilon^2 \int_{\R} \nonnorm \xi \odot \xi \,dx_1\\
            -\frac{\epsilon^2}{2} \int_\R (\nonnorm{\Theta} D_{\bar{x}}^2\Theta + \nontan{\Theta} D_{\bar{x}}^2 \Gamma)|_S\,dx_1.
        \end{multline}
        Thus, using Lemma~\ref{simplifiedaugVlemma}, we find
        \begin{multline*}
            \jb{D^2\augV(v)\dot{v},\dot{v}} =\int_{\R} \parn*{g +\mathfrak{b}_2'\mathfrak{b}_1}\dot{\eta}^2\,dx_1- \int_{\R} \parn*{\frac{b}{\jb{\eta'}^3}\dot{\eta}'}'\dot{\eta}\,dx_1-\int_{\R} \dot{\eta} \mathcal{M} \dot{\eta}\,dx_1 \\
            + 2\epsilon \dot{\bar{x}} \cdot \int_{\R} \dot{\eta}\mathfrak{b}_1\nontan(\DN(\eta)^{-1}\nonnorm \xi-\xi) \,dx_1\\
            + \dot{\bar{x}}^T \parn*{D_{\bar{x}}^2 \augHam(u_*) -\epsilon^2 \int_{\R} \nonnorm\xi \odot \DN(\eta)^{-1} \nonnorm\xi\,dx_1}\dot{\bar{x}},
        \end{multline*}
        which yields the claimed operator $A$.
    \end{proof}
    \begin{remark}
        \label{A33 simplification remark}
        Under natural symmetry assumptions on $v$, the expression for $A_{33}$ can be simplified further. Specifically, if $\eta$ is even and $\bar{x}_1 = 0$, then
        \[
            A_{33} = 2\epsilon^2 D_x^2 \Gamma_2(\bar{x}) - \epsilon \int_{\R} (\DN(\eta)\varphi_* D_{\bar{x}}^2 \Theta + \varphi_*' D_{\bar{x}}^2 \Gamma)|_S\,dx_1+ \epsilon^2\int_{\R} \nonnorm{\xi} \odot (\xi - \DN(\eta)^{-1}\nonnorm \xi)\,dx_1,
        \]
        and all three terms are diagonal matrices.
    \end{remark}
    
    We can now confirm that the augmented Hamiltonian admits an extension to the energy space.
    
    \begin{lemma}[Extension of $D^2\augHam$]
        \label{hamiltonian extension lemma}
        For all $ v \in \Vspace_{1,3} \cap \nbhdO_{1,3}$, there is a self-adjoint operator $\Hc(v) \in \Lin(\Xspace,\Xspace^*)$ such that
        \begin{equation}
            \jb{D^2\augHam(u_*(v))\dot{u},\dot{w}}_{\Vspace^* \times \Vspace} = \jb{\Hc(v)\dot{u},\dot{w}}_{\Xspace^* \times \Xspace}
            \label{X* representation D^2 E}
        \end{equation}
        for all $\dot{u},\dot{w} \in \Vspace$. The operator is given by
        \[
            \Hc(v)\dot{u} =
            \begin{pmatrix}
                \Id_{\Xspace_1^*} & 0 & 0\\
                0 & 0 & \Id_{\Xspace_2^*}\\
                0 & \Id_{\R^2} & 0
            \end{pmatrix}
            \begin{pmatrix}
                A(v) + \mathcal{L}(v)^*\DN(\eta)^{-1}\mathcal{L}(v) &   -\mathcal{L}(v)^*\\
                -\mathcal{L}(v) & \DN(\eta)
            \end{pmatrix}
            \begin{bmatrix}
                \dot{v}\\
                \dot{\varphi}
            \end{bmatrix},
        \]
        where $\mathcal{L}(v)$ and $A(v)$ are as defined in Lemmas~\ref{simplifiedaugVlemma} and \ref{characterization of augV lemma}, respectively. The adjoint $\mathcal{L}(v)^* \in \Lin(\Xspace_2;\Xspace_{1,3}^*)$ is given by
        \[
            \mathcal{L}(v)^*\dot{\varphi} = \left(\mathfrak{a}_2 \DN(\eta)\dot{\varphi}-\mathfrak{b}_1 \dot{\varphi}',\epsilon\jb{\nonnorm\xi,\dot{\varphi}}\right),
        \]
        and we have
        \begin{equation}
            \label{Hc quadratic form formula}
            \jb{\Hc \dot{u},\dot{u}}_{\Xspace^* \times \Xspace} = \jb{A(v)\dot{v},\dot{v}}_{\Xspace_{1,3}^*\times\Xspace_{1,3}} + \jb{\DN(\eta)(\dot{\varphi} - \DN(\eta)^{-1}\mathcal{L}\dot{v}),(\dot{\varphi} - \DN(\eta)^{-1}\mathcal{L}\dot{v})}_{\Xspace_2^* \times \Xspace_2}
        \end{equation}
        for all $\dot{u} \in \Xspace$.
    \end{lemma}
    \begin{proof}
        Again, we need only consider the diagonal. By Lemmas~\ref{characterization of augV lemma} and \ref{simplifiedaugVlemma} one has
        \begin{multline*}
            \jb{D^2\augHam(u_*(v))\dot{u},\dot{u}}_{\Vspace^* \times \Vspace} =\jb{A\dot{v} + \mathcal{L}(v)^*\DN(\eta)^{-1}\mathcal{L}(v)\dot{v},\dot{v}}_{\Xspace_{1,3}^*\times \Xspace_{1,3}}\\
            + \jb{D_\varphi^2 \augHam(u_*(v))\dot{\varphi}+2D_\varphi D_v \augHam(u_*(v))\dot{v},\dot{\varphi}}_{\Vspace_2^* \times \Vspace},
        \end{multline*}
        for all $\dot{u} \in \Vspace$, and it is simple to verify that
        \[
            \jb{D_\varphi D_v \augHam(u_*(v))\dot{v},\dot{\varphi}}_{\Vspace_2^* \times \Vspace_2} = - \jb{\mathcal{L}(v)\dot{v},\dot{\varphi}}_{\Xspace_2^* \times \Xspace_2}
        \]
        for all $\dot{v} \in \Vspace_{1,3}$ and $\dot{\varphi} \in \Vspace_2$.
    \end{proof}
    
    Using the representation for $D^2 \augV$ furnished by Lemma~\ref{hamiltonian extension lemma} in conjunction with the asymptotics derived in Appendix~\ref{existence theory appendix}, we are at last able to prove that Assumption~\ref{spectral assumptions} is satisfied.
    
    \begin{theorem}
        Let $\mathcal{I} \subset (0,\infty)$ be a nontrivial compact interval, and consider the family of bound states $\mathscr{C}_\mathcal{I}^\epsilon$ defined in \eqref{point vortex U_c}, furnished by Theorem~\ref{existence theorem}. Fix $0 < \abs{\epsilon} \ll 1$. Then the spectrum of $\Hc = H_{c(\epsilon,a)}(v(\epsilon,a))$ has the form
        \[
            \spectrum(I^{-1}\Hc) = \{-\mu_c^2\} \cup \{0\} \cup \Sigma_c
        \]
        for all $a \in \mathcal{I}$, with $-\mu_c^2 < 0$ and $0$ being simple eigenvalues, and $\Sigma_c \subset (0,\infty)$ bounded away from $0$.
    \end{theorem}
    \begin{proof}
        Under this hypothesis, we may view $\Hc$ as a small perturbation of the block diagonal operator
        \[
            \begin{pmatrix}
                g -b\partial_{x_1}^2 & 0 & 0\\
                0 & \abs{\partial_{x_1}} & 0\\
                0 & 0 & 0
            \end{pmatrix} \in \Lin(\Xspace,\Xspace^*),
        \]
        whose spectrum clearly consists of a part $\tilde{\Sigma} \subset (0,\infty)$ bounded away from $0$, plus the eigenvalue $0$ with multiplicity two. Thus the spectrum of $\Hc$ will have a part $\Sigma_c \subset (0,\infty)$ bounded away from $0$, plus two eigenvalues near the origin. We know that one of these is exactly $0$, with corresponding eigenvector $T'(0)U_c$.  Finally, from Lemma~\ref{formal stability lemma} and \eqref{abstract d'' identity} we see that $\frac{dU}{dc}$ is a negative direction for $\Hc$. Thus the other eigenvalue has to be negative.
    \end{proof}

    At this stage, we have completely verified that the myriad hypotheses of the abstract stability theory are satisfied for the solutions constructed in Appendix~\ref{existence theory appendix}. Theorem~\ref{main point vortex theorem}  therefore follows  immediately from Theorem~\ref{abstract stability theorem}.

\section{Stability for a class of dispersive model equations}
    \label{dispersive section}
    
    As a second illustration of the abstract theory, we devote this section to studying the stability properties of solitary wave solutions to the nonlinear dispersive PDE
    \begin{equation}
        \label{general model equation} 
        \partial_t u= \partial_x(\Lambda^\alpha u-u^p),
    \end{equation}
    where $u = u(t,x)\colon \R \times \R \to \R$ is the unknown, $\Lambda \coloneqq \abs{\partial_x}$, $\alpha \in (1/3,2]$, and 
    \begin{equation}
        \label{model p condition}
        p \in \N \cap \begin{cases}
            (1,(1+\alpha)(1-\alpha)^{-1}) & \alpha \in (1/3,1),\\
            (1,\infty) & \alpha \in [1,2].
        \end{cases}
    \end{equation}
    Heuristically, $\alpha$ describes the strength of the dispersion, while $p$ describes the strength of the nonlinearity.
    
    Equations of the general form \eqref{general model equation} include a number of extremely important hydrodynamical models. In particular, when $p=2$, the cases $\alpha = 1$ and $\alpha=2$ are known as the Benjamin--Ono equation (BO) and  Korteweg--de Vries equation (KdV), respectively. KdV, among many other things,  governs surface waves in shallow water. Benjamin--Ono models the motion of waves along the interface between two infinitely deep fluid regions in a certain long-wave regime \cite{benjamin1967internal,ono1975algebraic}.
    
    In \cite{bona1987stability}, Bona, Souganidis, and Strauss investigated the orbital stability and instability of solitary-wave solutions to \eqref{general model equation} for $\alpha \in [1,2]$.  Their strategy relied on many of the ideas underlying the GSS method.  However, as we will see below, the corresponding Poisson map $J$ was not surjective, and hence a number of adaptations were necessary.  Specifically, the authors made use of another conserved quantity --- the mass $\int u \, dx$ --- requiring them to obtain estimates on the spatial decay rates of solitary waves in order to ensure the persistence of integrability.  
    
    Our purpose in this section is to offer a new proof of the Bona, Souganidis, and Strauss theorem that follows directly from the stability machinery presented in Section~\ref{abstract stability section} and Section~\ref{abstract instability}.  Because we do not appeal to the mass, no asymptotic estimates are required. Notice also that we treat ``fractional'' dispersive model equations for which $\alpha \in (1/3,1)$.   Orbital stability results for $\alpha \in (1/2,1)$ have been obtained by Linares, Pilod, and Saut \cite{lineares2015remarks}, and Angulo Pava \cite{pava2018stability}; we discuss the connections between these works and the present paper further below.  Theorem~\ref{bss stability theorem} below gives conditional orbital instability for  fKdV ($p=2$) when $\alpha \in (1/3,1/2)$, and this appears to be new.  Indeed, Linares, Pilod, and Saut observe that the Bona, Souganidis, and Strauss approach \emph{almost} works in this regime, except that the tail estimates fail to hold.    
    
    While we do not pursue it here, one can also consider more general nonlinearities at the expense of some sharpness.  Another interesting possible extension is to study dispersive PDEs like the Whitham equation, where $\Lambda^\alpha$ in \eqref{general model equation} is replaced by a Fourier multiplier with an inhomogeneous symbol.         
    
    It is also important to note that, by specializing to specific choices of $\alpha$ and $p$, one can say much more.  As one example, for $\alpha = 2$ (gKdV), Pego and Weinstein \cite{pego1994asymptotic}, Mizumachi \cite{mizumachi2001largetime}, Martel and Merle \cite{martel2001asymptotic}, and Germain, Pusateri, and Rousset \cite{germain2016asymptotic} obtain asymptotic stability results (in different topologies) for various subcritical cases $p < 5$.  For supercritical waves $p > 5$, Jin, Lin, and Zeng \cite{jin2018dynamics} were able to completely classify the $H^1$ dynamics near the family of solitary waves using invariant manifold techniques.  The main appeal of our approach is its relative simplicity, and the fact that it simultaneously addresses the range of dispersion strengths $\alpha \in (1/3,2]$ and nonlinearities \eqref{model p condition}.
    
    \subsection{Reformulation as a Hamiltonian system}
        
        Formally, the expression inside the parentheses on the right-hand side of \eqref{general model equation} is the derivative of the energy
        \begin{equation}
            \label{model energy definition}
            \eng(u) \coloneqq \frac{1}{2} \int_{\R} (\Lambda^{\frac{\alpha}{2}} u)^2 \, dx - \frac{1}{p+1}\int_{\R} u^{p+1} \, dx.
        \end{equation}
        This suggests that the natural energy space is $\Xspace \coloneqq H^{\frac{\alpha}{2}}(\R)$, with the dual space $\Xspace^* = H^{-\frac{\alpha}{2}}(\R)$, and the isomorphism $I\colon\Xspace \to \Xspace^*$ given by $\jb{\Lambda}^\alpha$. The condition $\alpha > 1/3$ ensures the existence of admissible $p$, those satisfying \eqref{model p condition}, which in particular implies that $\Xspace \hookrightarrow L^{p+1}(\R)$. Observe that $E$ defined according to \eqref{model energy definition} then lies in $C^\infty(\Xspace;\R)$, and that indeed
        \[
            DE(u) =\Lambda^\alpha u - u^p
        \]
        for all $u \in \Xspace$.  We may therefore take
        \begin{equation}
            \label{model equation def V} 
            \Vspace \coloneqq \Xspace.
        \end{equation}
        
        The local and global well-posedness of the Cauchy problem for \eqref{general model equation} is still an active subject of research, and what is currently known depends considerably on $\alpha$ and $p$. To state things concisely, we suppose that \eqref{general model equation} is known to be locally well-posed in $H^{s}$ for $s > s_0 = s_0(\alpha,p)$, and set
        \begin{equation}
            \label{model equation def W}
            \Wspace \coloneqq \begin{cases}
                \Xspace & \text{if $\frac{\alpha}{2} > s_0(\alpha,p)$},\\
                H^{s_0+}(\R) & \text{if $\frac{\alpha}{2} \leq s_0(\alpha,p)$}.
            \end{cases}
        \end{equation}
        At present, the best known result when $p=2$ is $s_0(\alpha, 2) = 3/2-5\alpha/4$,
        and hence \eqref{general model equation} is \emph{globally} well-posed in $\Xspace$ when $\alpha > 6/7$ and $p=2$; see \cite{Molinet2015,Molinet2018}.  This is conjectured to hold for all $\alpha > 1/2$, which corresponds to the $L^2$ subcritical case.  
        
        However, for fKdV with $\alpha \in (1/3,6/7]$, the functional analytic setup in \eqref{model equation def W} will lead to a conditional stability or instability result.   This is essentially what is done by Angulo Pava in \cite[Theorem 1.1]{pava2018stability}, as well as Linares, Pilod, and Saut in \cite[Theorem 2.14]{lineares2015remarks}, who treat the range $\alpha \in (1/2,1)$.     We caution, however, that in both of these papers the definition of ``conditional stability'' is less conditional than ours: we require the solution to remain in the ball $\mathcal{B}_R^\Wspace$, while they only ask for it to exist.
        
        Next, define the Poisson map $J \colon \Dom(J) \subset \Xspace^* \to \Xspace$ by
        \begin{equation}
            \label{model J}
            J \coloneqq \partial_x,
        \end{equation}
        with domain $\Dom(J) \coloneqq H^{1+\frac{\alpha}{2}}(\R)$.  As $J$ is independent of state, it can be identified with $\hat J$ in Assumption~\ref{abstract symplectic assumption}. Moreover, $J$ is clearly injective, and skew-adjoint. The Cauchy problem for \eqref{general model equation} can now be restated rigorously as the abstract Hamiltonian system
        \begin{equation}
            \label{model hamiltonian equation}
            \frac{d}{dt} \jb{u(t),w}= \jb{u^p-\Lambda^\alpha u,\partial_x w} \quad \text{for all } w \in H^{1+\frac{\alpha}{2}}(\R),\quad u(0)=u_0,
        \end{equation}
        by specializing the general system in \eqref{weak abstract Hamiltonian system}.
        
        The equation \eqref{general model equation} possesses a number of symmetries, but the one of most interest to us is spatial translation invariance. For each $s \in \R$, we define $T(s) \in \Lin(\Xspace)$ by
        \begin{equation}
            \label{model T definition}
            T(s) u \coloneqq u(\placeholder-s),
        \end{equation}
        and this forms a group of unitary operators on $\Xspace$. Its infinitesimal generator is $T'(0)=-\partial_x$, with domain $\Dom(T'(0)) = H^{1+\frac{\alpha}{2}}(\R)$. Moreover, since $T'(0) = J(-\iota_{\Xspace \to \Xspace^*})$, the group generates the momentum
        \begin{equation}
            \label{model momentum}
            \mom(u) \coloneqq  \frac{1}{2}\jb{-u,u} = -\frac{1}{2} \int_{\R} u^2 \, dx,
        \end{equation}
        which also is of class $C^\infty(\Xspace; \R)$.
        
        The next lemma collects and expands upon these observations to confirm that the Hamiltonian formulation meets the requirements of the general theory.  
        
        \begin{lemma}
            \label{model hamiltonian lemma}
            The Hamiltonian formulation \eqref{model hamiltonian equation} of the dispersive model equation \eqref{general model equation} satisfies Assumptions~\ref{abstract interpolation assumption}-\ref{abstract symmetry assumption}.
        \end{lemma}
        \begin{proof}
            Because $\Vspace = \Xspace$, both Assumption~\ref{abstract interpolation assumption} and Assumption~\ref{extend DP and DE assumption} hold trivially.  Likewise, for $J$ defined as in \eqref{model J}, we have already verified that the relevant requirements of Assumption~\ref{abstract symplectic assumption} are met. To show that the symmetry group satisfies Assumption~\ref{abstract symmetry assumption} requires chasing the definitions. Both invariance and the commutativity are readily checked, as differentiation commutes with translation. Finally, the only remaining property that requires elaboration is \ref{range density}. We see that
            \begin{align*}
                \Dom(T'(0)|_{\Wspace}) \cap \Rng{J} &= \begin{cases}
                    H^{1+\frac{\alpha}{2}}(\R) \cap\partial_x H^{1+\frac{\alpha}{2}}(\R) & \quad \text{if $\frac{\alpha}{2} > s_0(\alpha,p)$}, \\
                    H^{(1+s_0)+}(\R) \cap\partial_x H^{1+\frac{\alpha}{2}}(\R) & \quad \text{if $\frac{\alpha}{2} \leq s_0(\alpha,p)$},
                \end{cases}
            \end{align*}
            which is dense in $\Xspace$ by the same kind of argument as in Lemma~\ref{sobolev density lemma}.
        \end{proof}

    \subsection{Solitary waves and spectral properties}
        
        It is well known that the dispersive models captured by \eqref{general model equation} support solitary waves $u(t)=T(ct)U_c=U_c(\placeholder-ct)$ for all $c > 0$. Recall that such $U_c \in \Xspace$ must satisfy
        \begin{equation}
            \label{model bound state equation}
            D\augHam(U_c) = \Lambda^\alpha U_c - U_c^p + cU_c = 0,\qquad (\text{in $\Xspace^*$})
        \end{equation}
        and by introducing the scaling
        \begin{equation}
            \label{model ground state scaling}
            U_c = c^{\frac{1}{p-1}}Q(c^{\frac{1}{\alpha}} \placeholder)
        \end{equation}
        we see that all such waves are just scaled versions of solutions of the equation
        \begin{equation}
            \label{model ground state equation}
            Q + \Lambda^\alpha Q = Q^p.
        \end{equation}
        
        \begin{lemma}
            If $Q \in \Xspace$ is a nontrivial solution of \eqref{model ground state equation}, then the family $\brac{U_c \colon c \in (0,\infty)}$ defined through \eqref{model ground state scaling} satisfies Assumption~\ref{bound state assumption}.
        \end{lemma}
        \begin{proof}
            By a standard bootstrapping argument, we have that any such solution $Q$ lies in $H^r(\R)$ for every $r \geq 0$. Since $U_c$ is defined by \eqref{model ground state scaling}, parts \ref{bound states improved regularity}-\ref{bound state domain assumption} are therefore immediate. Finally,
            \[
                \liminf_{\abs{s}\to \infty}{\norm{T(s)U_c - U_c}_{\Xspace}} = 2\norm{U_c}_{\Xspace} > 0,
            \]
            so the second option of part \ref{non-periodic bound state assumption} holds.
        \end{proof}

        It is also easily seen that if $Q \neq 0$ solves \eqref{model ground state equation},  then $Q$ is a critical point of the Weinstein functional $\mathcal{J} \in C^2(\Xspace \setminus \{0\};(0,\infty))$ defined by
        \[
            \mathcal{J}(u) \coloneqq \frac{\norm{u}_{\dot{H}^{\frac{\alpha}{2}}(\R)}^{\frac{p-1}{\alpha}} \norm{u}_{L^2(\R)}^{p+1 - \frac{p-1}{\alpha}}}{\norm{u}_{L^{p+1}(\R)}^{p+1}}.
        \]
        We say that a solution $Q$ of \eqref{model ground state equation}  is a ground state if $Q$ is not just a critical point of $\mathcal{J}$, but also an even, positive minimizer. Since $\mathcal{J}$ is invariant under scaling, the same is then true of each $U_c$, solving \eqref{model bound state equation}, defined through \eqref{model ground state scaling}. Note that the corresponding operator $\Hc \in \Lin(\Xspace,\Xspace^*)$ is given by
        \[
            \Hc u \coloneqq D^2E_c(U_c)u = \Lambda^\alpha u - pU_c^{p-1}u + cu,
        \]
        which is clearly self-adjoint. We have the following result, due to Frank and Lenzmann \cite{Frank2013}, vastly generalizing earlier results for KdV \cite{weinstein1982nonlinear} and BO \cite{amick1991uniqueness,amick1991benjamin}.
        
        \begin{lemma}
            \label{ground state existence lemma}
            There exists a unique ground state solution $Q \in \Xspace$ of \eqref{model ground state equation}. Moreover, for each $c \in (0,\infty)$, the spectrum of the operator $H_c$ corresponding to the bound state solution $U_c$ defined by \eqref{model ground state scaling} satisfies
            \[
                \spectrum{I^{-1}\Hc} = \{-\mu_c^2\} \cup \{0\} \cup \Sigma_c,
            \]
            where $-\mu_c^2 < 0$ and $0$ are simple eigenvalues, and $\Sigma_c \subset (0,\infty)$ is bounded away from zero. That is, Assumption~\ref{spectral assumptions} is satisfied.
        \end{lemma}
        
        Of course, $Q$ cannot be written down explicitly for most choices of $\alpha$ and $p$. Famously, for KdV
        \[
            Q_{\mathrm{KdV}}(x) = \frac{3}{2} \sech^2{\parn*{\frac{x}{2}}},
        \]
        while Benjamin \cite{benjamin1967internal} exhibited the ground state  
        \[
            Q_{\mathrm{BO}}(x) = \frac{2}{1+x^2}
        \]
        for BO in his original paper on the topic.
        
    \subsection{Stability and instability} 
        
        The analysis of the previous subsection confirms that the family $\{U_c : c \in (0,\infty)\}$ corresponding to the unique ground state $Q$ of \eqref{model ground state equation} furnished by Lemma~\ref{ground state existence lemma} falls into the scope of the general stability theory developed in Section~\ref{abstract stability section} and Section~\ref{abstract instability}.  We therefore obtain the following extended version of the classical result of Bona, Souganidis, and Strauss~\cite{bona1987stability}:
        
        \begin{theorem}
            \label{bss stability theorem}
            If $ p < 2\alpha +1$, then each solitary wave in the family $\{U_c : c \in (0,\infty)\}$ is conditionally orbitally stable in the sense of Theorem~\ref{abstract stability theorem} when $\frac{\alpha}{2} \leq s_0(\alpha,p)$, and orbitally stable in the sense of Corollary~\ref{X = W stability corollary} when $\frac{\alpha}{2} > s_0(\alpha,p)$. When $p > 2\alpha + 1$, the solitary waves are orbitally unstable in the sense of Theorem~\ref{abstract instability theorem}.
        \end{theorem}
        \begin{proof}
            Whether $U_c$ is stable or not reduces to the sign of $d''(c)$, where we recall that $d(c) \coloneqq \augHam(U_c)$ is the moment of instability. Exploiting the scaling \eqref{model ground state scaling} and the identity \eqref{abstract d' identity}, we find
            \[
                d'(c) = -P(U_c) = \frac{1}{2}\int_\R U_c^2\,dx = \frac{1}{2} c^{\frac{2}{p-1}-\frac{1}{\alpha}}\norm{Q}_{L^2(\R)}^2,
            \]
            whence
            \[
                \signum{d^{\prime\prime}(c)} = \signum{\left(\frac{2}{p-1} - \frac{1}{\alpha} \right)}
                \begin{cases}
                    > 0 & \text{if $p < 2\alpha + 1$,}\\
                    < 0 & \text{if $p > 2\alpha + 1$,} 
                \end{cases}
            \]
            which gives the statement in the theorem.
        \end{proof}

\section*{Acknowledgements}
    
    KV acknowledges the support by grants nos.~231668 and 250070 from the Research Council of Norway.
    The research of EW is supported by the Swedish Research Council (Grant nos.~621-2012-3753 and 2016-04999).
    The research of SW is supported in part by the National Science Foundation through awards DMS-1549934 and DMS-1812436.
    
    The authors also wish to thank Christopher Curtis, John Grue, and Chongchun Zeng for enlightening discussions related to this work.  

\appendix

\section{Function spaces}
    Define the Schwartz class $\mathcal{S}(\R)$ to be the set of all $f \in C^\infty(\R)$ such that $x^n f^{(m)}(x) \in L^\infty(\R)$ for all $n,m \in \N_0$, and also the subspace $\mathcal{S}_0(\R)$ of those $f \in \mathcal{S}(\R)$ for which $\hat{f}^{(n)}(0)=0$ for all $n \in \N_0$. For every $s \in \R$, we define the inhomogeneous Sobolev space $H^s(\R)$ to be the completion of $\mathcal{S}(\R)$ with respect to
    \begin{equation}
        \label{inhomogeneous norm}
        \norm{f}_{H^s(\R)} \coloneqq \norm{\jb{\placeholder}^s\hat{f}}_{L^2(\R)},
    \end{equation}
    which can be realized as the space of all $f \in \mathcal{S}'(\R)$ for which $\hat{f} \in L_{\text{loc}}^1(\R)$ and $\norm{f}_{H^s(\R)} < \infty$.
    
    The homogeneous Sobolev space $\dot{H}^s(\R)$, on the other hand, is defined to be the completion of $\mathcal{S}_0(\R)$ with respect to
    \begin{equation}
        \label{homogeneous norm}
        \norm{f}_{\dot{H}^s(\R)} \coloneqq \norm{\abs{\placeholder}^s\hat{f}}_{L^2(\R)} < \infty,
    \end{equation}
    and for $s < 1/2$ it can be realized like before as the space of all $f \in \mathcal{S}'(\R)$ for which $\hat{f} \in L_{\text{loc}}^1(\R)$ and $\norm{f}_{\dot{H}^s(\R)} < \infty$. If $s \geq 1/2$, set $n \coloneqq \lfloor s + 1/2 \rfloor$, so that $s = n + \alpha$ with $\alpha \in [-1/2,1/2)$. Then $\dot{H}^s(\R)$ can be realized as the space of all $f \in \mathcal{S}'(\R)$ such that $f^{(n)} \in \dot{H}^\alpha(\R)$, \emph{modulo} polynomials of degree at most $n-1$, with the norm \eqref{homogeneous norm} interpreted as $\norm{f^{(n)}}_{\dot{H}^\alpha}$.
    
    On domains $\Omega \subset \R^2$, we shall only have use for $\dot{H}^1(\Omega)$, defined as the space of $f \in L_{\text{loc}}^1(\Omega)/\R$ for which $\nabla f \in L^2(\Omega)$.
    \begin{lemma}[Density]
        \label{sobolev density lemma}
        For all $s,r \in \R$, the space $H^s(\R) \cap \dot{H}^r(\R)$ is dense in both $H^s(\R)$ and $\dot{H}^r(\R)$.
    \end{lemma}
    \begin{proof}
        Define $\chi_n \coloneqq \chi_{1/n<\abs{\xi}<n}$ for all $n \in \N$. Suppose first that $f \in H^s(\R)$, and set $f_n \coloneqq \mathscr{F}^{-1}(\chi_n\hat{f})$ for $n \in \N$. Then $f_n \in H^s(\R)\cap\dot{H}^r(\R)$ and
        \[
            \norm{f-f_n}_{H^s(\R)} = \norm{\jb{\placeholder}^s(1-\chi_n)\hat{f}}_{L^2(\R)},
        \]
        so $f_n \to f$ in $H^s(\R)$. Next, suppose that $f \in \dot{H}^r$ and choose $k \in \N$ large enough so that $r - k < 1/2$. Then the sequence
        \[
            f_n \coloneqq \mathscr{F}^{-1}\parn*{\chi_n \widehat{f^{(k)}} / (i\xi)^k} \in H^s(\R)\cap\dot{H}^r(\R)
        \]
        converges to $f$ in $\dot{H}^r(\R)$.
    \end{proof}

\section{Existence theory}
    \label{existence theory appendix}
    In this appendix, we present a slightly modified version of the existence theory for capillary-gravity waves with a point vortex  due to Shatah, Walsh, and Zeng \cite{shatah2013travelling}. The original paper fixes the location of the vortex, which is ill-suited for us. Since $\epsilon$ appears in the Poisson map, it must be held fixed on any family of waves to which we wish to apply the general stability theory. We can obtain such families by also allowing the location of the point vortex to vary.
    
    Specifically, we suppose that the point vortex is situated at $\bar{x} = -ae_2$, where $a > 0$. A symmetric traveling wave solution to \eqref{SWZ formulation} having wave speed $c$ must then satisfy the abstract operator equation 
    \begin{equation}
        \label{operator equation}
        \mathscr{F}(\eta,\varphi,c;\epsilon,a) = 0,
    \end{equation}
    with $\mathscr{F} = (\mathscr{F}_1, \mathscr{F}_2,\mathscr{F}_3) \colon O \subset (X \times \R \times (0,\infty)) \to Y$ defined by
    \begin{align*}
        \mathscr{F}_1(\eta,\varphi,c;\epsilon,a) &\coloneqq \begin{multlined}[t]\frac{(\varphi')^2 -2\eta'\varphi'\DN(\eta)\varphi - (\DN(\eta)\varphi)^2}{2\jb{\eta'}^2}-c\varphi'+g\eta - b\parn*{\frac{\eta'}{\jb{\eta'}}}' \\
        +\epsilon\varphi'\Theta_{x_1}|_S + \frac{\epsilon^2}{2}\parn*{\abs{\nabla \Theta}^2}|_S-\epsilon c \Theta_{x_1}|_S
        \end{multlined}\\
        \mathscr{F}_2(\eta,\varphi,c;\epsilon,a) &\coloneqq c\eta' + \DN(\eta)\varphi + \epsilon \nonnorm \Theta,\\
        \mathscr{F}_3(\eta,\varphi,c;\epsilon,a) &\coloneqq c - (\mathcal{H}(\eta)\varphi)_{x_1}(0,-a) + \frac{\epsilon}{4\pi a},
    \end{align*}
    where $\mathcal{H}(\eta)$ denotes the harmonic extension operator. We will use the spaces
    \begin{align*}
        X &\coloneqq H_e^k(\R) \times (\dot{H}_o^k(\R) \cap \dot{H}_o^{1/2}(\R)) \times \R,\\
        Y &\coloneqq H_e^{k-2}(\R) \times (\dot{H}_o^{k-1}(\R)\cap \dot{H}_o^{-1/2}) \times \R
    \end{align*}
    for any $k > 3/2$ fixed, with the subscripts indicating odd and even, and the open set
    \[
        O \coloneqq \{(\eta,\varphi,c;\epsilon,a) \in X \times \R \times (0,\infty) : \abs{\eta(0)} < a\}.
    \]
    The map $\mathscr{F}$ is then $C^\infty$ (even analytic), and we have the following existence theorem.
    
    \begin{theorem}
        \label{existence theorem}
        There exists a $C^\infty$-surface
        \[
            \{(\eta(\epsilon,a), \varphi(\epsilon,a);\epsilon,a) : (\epsilon,a) \in U\} \subset O \times \R \times (0,\infty)
        \]
        of solutions to \eqref{operator equation}, with $U$ an open neighborhood of $\{0\} \times (0,\infty)$. Asymptotically, the solutions are of the form
        \begin{equation}
            \label{asymptotic form waves}
            \begin{aligned}
                \eta(\epsilon,a) &= \epsilon^2 \eta_2(a)+ O(\epsilon^4),\\
                \varphi(\epsilon,a) &= O(\epsilon^3),\\
                c(\epsilon,a) &= \epsilon c_1(a) + O(\epsilon^3),
            \end{aligned}
        \end{equation}
        in $C_{\text{loc}}^1((0,\infty);X)$, with
        \begin{equation}
            \label{c1 and eta2 definition}
            c_1(a) \coloneqq - \frac{1}{4\pi a}, \qquad \eta_2(a) \coloneqq \frac{1}{4\pi^2}(g-b\partial_{x_1}^2)^{-1}\parn*{\frac{x_1^2-a^2}{(x_1^2+a^2)^2}}.
            \qedhere
        \end{equation}
    \end{theorem}
    \begin{proof}
        As in \cite{shatah2013travelling}, this result follows from the implicit function theorem applied to $\mathscr{F}$ at the trivial solutions $(0;0,a)$ for $a \in (0,\infty)$. We compute
        \[
            D_X\mathscr{F}(0;0,a) = \begin{pmatrix}
                g-b\partial_{x_1}^2 & 0 & 0\\
                0 & \abs{\partial_{x_1}} & 0\\
                0 & -(\mathcal{H}(0)\placeholder)_{x_1}(0,-a) & 1
            \end{pmatrix},
        \]
        which is clearly an isomorphism $X \to Y$, as it is a lower diagonal matrix with isomorphisms on the diagonal. Finally, the asymptotic expansions listed in \eqref{asymptotic form waves} can be found by implicit differentiation.
    \end{proof}

    It is possible to write the leading order surface term $\eta_2$ in terms of the so-called exponential integral $E_1$.
    
    \begin{theorem}
        Define the holomorphic function $f \colon \C \setminus(-\infty,0] \to \C$ by
        \[
            f(z) \coloneqq -e^z E_1(z) = (\gamma + \log(z))e^z - \sum_{k=1}^\infty \frac{H_k}{k!}z^k,
        \]
        where $\gamma$ is the Euler--Mascheroni constant and $H_k$ is the $k$-th harmonic number. If we write $w = x +i\alpha = \sqrt{g/b}(x_1+ia)$, then
        \[
            \eta_2(x_1) = \frac{1}{4\pi^2b}\tilde{\eta}_2\parn*{x}, \qquad \tilde{\eta}_2(x) \coloneqq \realpart{\parn*{\frac{f(w)+f(-w)}{2}}}.
        \]
        More explicitly,
        \begin{multline*}
            \tilde{\eta}_2(x) = \brak*{\parn*{\gamma + \log\parn*{\sqrt{x^2+\alpha^2}}}\cos(\alpha) - \frac{\pi}{2}\sin(\alpha)}\cosh(x)\\
            + \sin(\alpha)\arctan(x/\alpha)\sinh(x) - \sum_{k=1}^\infty \frac{H_{2k}}{(2k)!}(x^2+\alpha^2)^k T_{2k}\parn*{\frac{x}{\sqrt{x^2+\alpha^2}}}
        \end{multline*}
        for all $x \in \R$, with $T_k$ being the $k$-th Chebyshev polynomial.
    \end{theorem}
    \begin{proof}
        By using \eqref{c1 and eta2 definition} and the scaling, we see that $\tilde{\eta}_2$ solves the differential equation
        \[
            \tilde{\eta}_2(x) - \tilde{\eta}_2''(x) = \frac{x^2-\alpha^2}{(x^2+\alpha^2)^2} = \realpart{w^{-2}},
        \]
        and one may directly verify that $g(z) \coloneqq (f(z)+f(-z))/2$ satisfies $g(z)-g''(z)=z^{-2}$ in $\C \setminus \R$. Moreover, this is the unique solution that vanishes at infinity, as it can be shown using a well-known asymptotic series for $E_1$
         that $g(z) = 1/z^2 + O(z^{-4})$ as $\abs{z} \to \infty$.
    \end{proof}

\section{Derivatives of the energy and momentum}
    \label{variations appendix} 
    
    We record here the derivatives of $\eng$ and $\mom$ up to order two.  Fix $u = (\eta, \varphi, \bar{x}) \in \Vspace \cap \nbhdO$, and let $\dot u = (\dot \eta, \dot\varphi, \dot{\bar{x}}) \in \Vspace$ represent a variation. Some of the integrals must be understood in the dual-pairing sense. To simplify the notation, we also introduce
    \[
        \mathfrak{a} \coloneqq (\nabla \varphi_{\mathcal{H}})|_S, \qquad \xi \coloneqq (\Theta_{x_1},\Xi_{x_2}) = - \nabla_{\bar{x}} \Theta,
    \]
    and note that
    \[
        D_{\bar{x}}^2 \Theta = \begin{pmatrix}
            \Theta_{x_1 x_1} & \Xi_{x_1 x_2}\\
            \Xi_{x_1 x_2} & \Theta_{x_2 x_2}
        \end{pmatrix}.
    \]

    \subsection*{Variations of $\kinE_0$}
        From \eqref{defK} we compute
        \begin{align*}
            \jb{D_\varphi \kinE_0(u),\dot\varphi} &= \int_{\R} \dot\varphi \DN(\eta) \varphi \, dx_1,\\
            \jb{D_\eta \kinE_0(u),\dot\eta} &= \frac{1}{2} \int_{\R} \varphi \jb{D_\eta \DN(\eta)\dot\eta, \varphi} \, dx_1 = \int_{\R} \dot{\eta}\parn*{\frac{1}{2}\abs{\mathfrak{a}}^2 - \mathfrak{a}_2 \DN(\eta)\varphi}\,dx_1,
        \end{align*}
        with second variations
        \begin{align*}
            \langle D_\varphi^2 \kinE_0(u)\dot\varphi, \dot\varphi \rangle  &= \int_{\R} \dot\varphi \DN(\eta) \dot\varphi \, dx_1, \\
            \langle D_\varphi D_\eta \kinE_0(u) \dot\varphi, \dot\eta \rangle  &= \int_{\R} \dot\varphi \langle D_\eta \DN(\eta)\dot\eta, \varphi\rangle \, dx_1 = \int_{\R} \dot{\eta}(\mathfrak{a}_1 \dot{\varphi}' - \mathfrak{a}_2 \DN(\eta)\dot{\varphi})\,dx_1,\\
            \langle D_\eta^2 \kinE_0(u) \dot\eta, \dot\eta \rangle  &= \frac{1}{2} \int_{\R} \varphi \langle \langle D_\eta^2 \DN(\eta)\dot\eta, \dot\eta \rangle,  \varphi \rangle \, dx_1 = \int_{\R} (\mathfrak{a}_1'\mathfrak{a}_2 \dot{\eta}^2 + \mathfrak{a}_2 \dot{\eta}\DN(\eta)(\mathfrak{a}_2\dot{\eta}))\,dx_1,
        \end{align*}
        where the explicit expressions for the shape-derivatives only hold when $\varphi \in \Xspace_2^{3/2}$.
    
    \subsection*{Variations of $\kinE_1$}
        From \eqref{defK} we find quickly that
        \[ 
            \jb{D_\eta \kinE_1(u),\dot\eta}  = \int_{\R} \dot \eta \varphi' \Theta_{x_1}|_S \, dx_1, \quad
        \jb{D_\varphi \kinE_1(u),\dot\varphi} = \int_{\R} \dot\varphi \nonnorm \Theta  \, dx_1,
        \]
        and
        \[
            \nabla_{\bar{x}} K_1(u) = -\int_{\R} \varphi \nonnorm \xi\,dx_1.
        \]
        The second variations are thus
        \begin{align*}
            \jb{ D_\eta^2 \kinE_1(u) \dot\eta, \dot\eta }  &=  \int_{\R} \dot\eta^2 \varphi'  \Theta_{x_1 x_2}|_S \, dx_1, &\qquad 
            \jb{ D_\eta D_\varphi \kinE_1(u)\dot\eta, \dot\varphi}  &=  \int_{\R} \dot\eta \dot\varphi'  \Theta_{x_1}|_S \, dx_1,\\
            \nabla_{\bar{x}} \jb{D_\eta \kinE_1(u),\dot{\eta}} &= - \int_{\R} \dot{\eta} \varphi' \xi_{x_1}|_S\,dx_1, &\qquad
            \nabla_{\bar{x}} \jb{D_\varphi \kinE_1(u),\dot{\varphi}} &= - \int_{\R} \dot{\varphi} \nonnorm \xi\,dx_1,\\
        \end{align*}
        and
        \[
            D_{\bar{x}}^2 \kinE_1(u) = \int_{\R} \varphi \nonnorm D_{\bar{x}}^2 \Theta\,dx_1.
        \]
        Note that in the above computations we have made repeated use of the fact that $\Theta$ is harmonic in a neighborhood of the surface $S$.  In particular, this implies that
        \begin{equation}
            \label{Theta identities}
            \nonnorm \Theta_{x_1}  = \nontan \Theta_{x_2} = (\Theta_{x_2}|_S)', \qquad  \nonnorm \Theta_{x_2} = -\nontan\Theta_{x_1} = - ( \Theta_{x_1}|_S)'.
        \end{equation}
        Similar identities hold for $\Xi$ as well. 
        
    \subsection*{Variations of $\kinE_2$}
        From \eqref{defK} we find
        \[
            \jb{D_\eta \kinE_2(u),\dot\eta} = \frac{1}{2} \int_{\R}\dot\eta (\abs{\nabla \Theta}^2)|_S  \, dx_1, \qquad \nabla_{\bar{x}} \kinE_2(u) =\nabla\Gamma_2(\bar{x}) - \frac{1}{2}\int_{\R} \nonnorm(\Theta \xi)\,dx_1.
        \]
        The second variations are thus
        \begin{align*}
            \jb{D_\eta^2 \kinE_2(u) \dot{\eta},\dot{\eta}} &= \int_{\R} (\nabla \Theta \cdot \nabla \Theta_{x_2})|_S\dot{\eta}^2\,dx_1,\\
            \nabla_{\bar{x}} \jb{D_\eta \kinE_2(u),\dot\eta} & = - \int_{\R} \dot{\eta}((D_x\xi)\nabla\Theta)|_S\,dx_1,\\
            D_{\bar{x}}^2 \kinE_2(u) &= 2D_x^2 \Gamma_2(\bar{x}) + \frac{1}{2} \int_{\R} \nonnorm(\Theta D_{\bar{x}}^2\Theta + \xi \xi^T)\,dx_1.
        \end{align*}
        
    \subsection*{Variations of $\potE$}
        From \eqref{defV} we have
        \begin{align*}
            \jb{D_\eta \potE(u),\dot{\eta}} &= \int_{\R} \parn*{g\eta - b \parn*{\frac{\eta'}{\jb{\eta'}}}'}\dot{\eta}\,dx_1,\\
            \jb{D_\eta^2 \potE(u)\dot{\eta},\dot{\eta}} &= \int_{\R} \parn*{g\dot{\eta}^2 + b\frac{1}{\jb{\eta'}^3}(\dot{\eta}')^2}\,dx_1.
        \end{align*}
        
    \subsection*{Variations of $\mom$}
        Lastly we consider the momentum.  The first variations are given by
        \[
         \jb{D_\eta \mom(u),\dot \eta}  = \int_{\R} \dot \eta ( \varphi' + \epsilon \Theta_{x_1}|_S) \, dx_1, \qquad 
         \jb{D_\varphi \mom(u),\dot \varphi}    = -\int_{\R} \eta^\prime \dot{\varphi} \, dx_1,\]
        and
        \[
            \nabla_{\bar{x}} P(u) = \epsilon e_2 + \epsilon\int_{\R} \eta' \xi|_S\,dx_1.
        \]
        Likewise, we find that the second variations are
        \begin{align*}
            \jb{ D_\eta^2 \mom(u) \dot{\eta}, \dot{\eta}} &= \epsilon\int_{\R} \dot{\eta}^2  \Theta_{x_1 x_2}|_S \, dx_1, & \qquad
            \jb{D_\eta D_\varphi \mom(u) \dot \eta, \dot \varphi} &= - \int_{\R} \dot{\eta}' \dot\varphi \, dx_1,\\
            \nabla_{\bar{x}} \jb{D_\eta \mom(u),\dot{\eta}} &= - \epsilon \int_{\R} \dot{\eta} \xi_{x_1}|_S\,dx_1, & \qquad
            D_{\bar{x}}^2 \mom(u) &= -\epsilon\int_{\R}\eta' (D_{\bar{x}}^2 \Theta)|_S\,dx_1.
        \end{align*}

\bibliographystyle{siam}
\bibliography{pt_vortex}

\end{document}